\documentclass[11pt,reqno]{amsart} 
\usepackage{amsthm}
\usepackage{amsmath}
\usepackage{amsfonts}
\usepackage{amssymb}
\usepackage{xypic}
\usepackage{graphicx}
\usepackage{url}
\usepackage{enumerate}
\usepackage{lscape}
\usepackage[usenames,dvipsnames]{color}
\usepackage{multicol}

\usepackage{xcolor}

\newtheorem{theorem}{Theorem}[section]
\newtheorem{corollary}[theorem]{Corollary}
\newtheorem{proposition}[theorem]{Proposition}
\newtheorem{lemma}[theorem]{Lemma}
\newtheorem{example}[theorem]{Example}

\newtheorem{definition}[theorem]{Definition}
\newtheorem{condition}[theorem]{Condition}
\newtheorem*{question*}{Question} 
\newtheorem*{problem*}{Problem} 
\theoremstyle{remark}
\newtheorem{remark}[theorem]{Remark}

\numberwithin{equation}{section}
\newcommand{\Rplus}{\mathbb{R}_{>0}}
\newcommand{\Rnn}{\mathbb{R}_{\geq 0}}

\newcommand{\Etot}{E_{\mbox{tot}}}
\newcommand{\Ftot}{F_{\mbox{tot}}}
\newcommand{\Stot}{S_{\mbox{tot}}}

\newcommand{\arrowschem}[2]{\raisebox{-2ex}%
	{$\stackrel{\stackrel{\displaystyle#1}{\longrightarrow}}%
	{\stackrel{\longleftarrow}{#2}}$}}
\newcommand{\s}{s}
\newcommand{\es}{c}
\newcommand{\fs}{d}
\newcommand{\e}{e}
\newcommand{\f}{f}

\newcommand{\Sys}{\Sigma}
\newcommand{\Edg}{E}
\newcommand{\Egen}{M}
\newcommand{\St}{\mathcal{S}}
\newcommand{\supp}{{\rm supp}}
\newcommand{\CSMatrix}{complex-to-species rate matrix}
\newcommand{\CRN}{chemical reaction network }

\newcommand{\CRNsNoSpace}{chemical reaction networks}

\newcommand{\CRSsCapital}{CHEMICAL REACTION SYSTEMS }
\newcommand{\CRS}{chemical reaction system }
\newcommand{\CRSs}{chemical reaction systems }
\newcommand{\CRSNoSpace}{chemical reaction system}
\newcommand{\CRSsNoSpace}{chemical reaction systems}

\newcommand{\invtPoly}{\mathcal{P}_{x^0}}

\newcommand{\R}{\mathbb{R}}
\newcommand{\Q}{\mathbb{Q}}

\definecolor{dgreen}{rgb}{.2,.6,.2}
\colorlet{darkgreen}{black!30!dgreen}

\definecolor{dblue}{rgb}{0.0,0.0,0.68}

\DeclareMathOperator{\im}{im}
\DeclareMathOperator{\sign}{sign}
\DeclareMathOperator{\diag}{diag}

\begin{document} 

\title{\CRSsCapital with toric steady states}
\author{Mercedes P\'erez Mill\'an, Alicia Dickenstein, Anne Shiu, and
  Carsten Conradi}
{\address{
       MPM: Dto.\ de Matem\'atica, FCEN, Universidad de Buenos Aires, 
    Ciudad Universitaria, Pab.\ I, C1428EGA Buenos Aires, Argentina.  
     AD: Dto.\ de Matem\'atica, FCEN, Universidad de Buenos Aires, 
    Ciudad Universitaria, Pab.\ I, C1428EGA Buenos Aires, Argentina, and IMAS/CONICET.
    AS: Dept.\ of Mathematics,
    Duke University, Box 90320, Durham NC 27708-0320, USA.
	CC: Max-Planck-Institut Dynamik komplexer technischer Systeme, 
    Sandtorstr.\ 1, 39106 Magdeburg, Germany.}}
\footnote{MPM and AD were
  partially supported by UBACYT X064, CONICET PIP 112-200801-00483,
  and ANPCyT PICT 2008-0902, Argentina.  AS was supported by the NSF
  (DMS-1004380).  CC was supported by ForSys/MaCS 0313922. }

\email{mpmillan@dm.uba.ar, alidick@dm.uba.ar, annejls@math.duke.edu, conradi@mpi-magdeburg.mpg.de}


\begin{abstract}
Mass-action chemical reaction systems are frequently used in
  Computational Biology. The corresponding polynomial
  dynamical systems are often large (consisting of tens or even hundreds
  of ordinary differential equations) and poorly parametrized
  (due to noisy measurement data and a
  small number of data points and repetitions). Therefore, it is often
  difficult to establish the existence of (positive) steady states 
  or to determine whether more complicated phenomena such as multistationarity exist.
 If, however, the steady state ideal of the system is a binomial ideal, then 
  we show that these questions can be answered easily.  The focus of this work 
  is on systems with this property, and we say that such systems have toric steady states.  
  Our main result gives sufficient conditions for a \CRS to have 
  toric steady states.  Furthermore, we analyze the capacity of such a system to exhibit
  positive steady states and multistationarity. 
  Examples of systems with toric steady states include 
weakly-reversible zero-deficiency chemical reaction systems.
An important application of our work concerns the networks that describe  
the multisite phosphorylation of a protein by a
kinase/phosphatase pair in a sequential and distributive mechanism.

\end{abstract}

\maketitle

\section{Introduction}\label{sec:intro}

  Ordinary differential equations (ODEs) are an important modeling
  tool in Systems Biology and many other areas of Computational Biology. Due
  to the inherent complexity of biological systems, realistic
  models are often large, both in terms of the number of states and
  the (unknown) parameters. Moreover, models are often poorly
   parametrized, a consequence of noisy measurement data, a small
  number of data points, and a limited number of repetitions. Hence,
  for mass-action chemical reaction systems, the focus of the present 
  article, simply
  establishing the existence of (positive) steady states can be
  demanding, as it requires the solution of a large polynomial system
  with unknown coefficients (usually the parameters). Moreover, due to
  the predominant parameter uncertainty, one is often not interested in
  establishing the existence of a particular steady state, but rather
  in obtaining a  parametrization of all steady states -- preferably
  in terms of the system parameters \cite{TG}. Frequently one is also
  interested in the existence of multiple steady states
  (multistationarity), for example, in modeling the cell cycle
  \cite{cyc-005,cyc-004,cyc-017}, signal transduction
  \cite{sig-042,Markevich04} or cellular differentiation
  \cite{bif-006,bif-013}. For general polynomial systems with unknown
  coefficients, the tasks of obtaining positive solutions or a
   parametrization of positive solutions, and deciding about multiple
  positive solutions, are clearly challenging. For the systems
  considered in this article -- \CRSs with {\em toric steady states} -- these
  questions can be answered easily. 

  We say that a polynomial dynamical system $dx/dt = f(x)$
  has {\em toric steady states} if the ideal generated by its steady
  state equations is a binomial ideal (see
  Definition~\ref{def:ss}). We give sufficient conditions for a \CRS
  to have toric steady states (Theorems~\ref{th:toric} and~\ref{thm:necessary}) 
  and show in this case that the steady
  state locus has a nice monomial parametrization
  (Theorems~\ref{thm:toric_param} and~\ref{thm:toricgeneral_param}). 
  Furthermore, we show that the
  existence of positive steady states in this case is straightforward
  to check (Theorem~\ref{thm:multi_ss_toric}).
 
There are several important classes of mass-action kinetics \CRSs which
have toric steady states. These include usual instances of
detailed-balanced systems in the sense of Feinberg, Horn, and Jackson
\cite{Feinberg72,Feinberg89,Horn72,HornJackson}, which show
particularly nice dynamical behavior. These systems are
weakly-reversible, a hypothesis we do not impose here.

A \CRS with toric steady states of great biological importance is
the {\em multisite phosphorylation system}; 
this network describes the $n$-site phosphorylation of a protein by a
  kinase/phosphatase pair in a sequential and distributive
  mechanism. Biochemically, these systems play an important role in
  signal transduction networks, cell cycle control, or cellular
  differentiation: for example, members of the family of 
  mitogen-activated kinase cascades consist of several such phosphorylation
  systems with $n=2$ or $n=3$ (see e.g.~\cite{sig-016,sig-051}), the
  progression from G1 to S phase in the cell cycle of 
  budding yeast is controlled by a system with $n=9$ (by way of the protein
  Sic1, see e.g. \cite{cyc-007}), and a system with $n=13$ plays an
  important role in T-cell differentiation (by way of the protein NFAT
  \cite{NFAT-002,NFAT-001,NFAT-003}).

Consequently there exists a body of work on the mathematics of
  phosphorylation systems and the more general class of 
  post-translational modification systems: for example, Conradi {\em
    et al.}~\cite{ConradiUsing}, Wang and Sontag \cite{WangSontag}, 
  Manrai and Gunawardena~\cite{ManraiGuna}, and Thomson and Gunawardena
  \cite{TG,TG2}. While the first two references are concerned with the
  number of steady states and multistationarity, the references of
  Gunawardena {\em et al.}\ deal with  parametrizing all positive
  steady states. The present article builds on these earlier results.
  In fact, the family of monomial parametrizations
  obtained here for multisite
  phosphorylation systems (Theorem~\ref{thm:n-site}) 
  is a specific instance of a rational
  parametrization theorem due to Thomson and Gunawardena, and one
  parametrization of the family was analyzed earlier by Wang and
  Sontag. Furthermore, we show that by using results from \cite{ConradiUsing}
  one can determine whether
  multistationarity exists for systems with toric steady states by analyzing
  certain linear inequality systems. In this sense our results can be seen
  as a generalization of \cite{ConradiUsing}.

This article is organized as follows.  Section~\ref{sec:intro_CRN} provides an
introduction 
to the mathematics of chemical reaction systems.  Our main results on toric
steady states appear in Section~\ref{sec:toric}: Theorems~\ref{th:toric} and~\ref{thm:necessary}
give sufficient criteria for a system to exhibit toric steady states, and 
Theorems~\ref{thm:toric_param} and~\ref{thm:toricgeneral_param} give parametrizations 
for the steady state locus.
As an application of this work, we analyze the steady state loci of multisite
phosphorylation systems in Section~\ref{sec:multisite_toric}.  Theorem~\ref{thm:n-site}
summarizes our results: we show that these systems have toric
steady states for any choice of reaction rate constants, and we give an explicit parametrization
of the steady state locus.  
Section~\ref{sec:multistat} focuses on multiple steady states for chemical
reaction systems with toric steady states. Theorem~\ref{thm:multi_ss_toric} 
gives a criterion for such a system to exhibit multistationarity, and we make the 
connection to a related criterion due to Feinberg.

\section{Chemical reaction network theory}\label{sec:intro_CRN}

In this section we recall the basic setup of chemical reaction systems,
and we introduce in \S~\ref{sec:steadyStes} the precise definition of
systems with toric steady states. 
 We first present an intuitive example that
illustrates how a chemical reaction network gives rise to a dynamical system. 
An example of a {\em chemical reaction}, as it usually appears in the
literature, is the following:
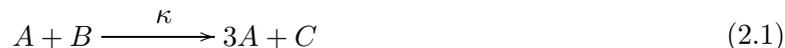
\begin{align} \label{reaction_ex}
 \begin{xy}<15mm,0cm>:
 (1,0) ="3A+C" *+!L{3A+C} *{};
 (0,0) ="A+B" *+!R{A+B} *{};
 (0.55,.05)="k" *+!D{\kappa} *{};
   {\ar "A+B"*{};"3A+C"*{}};
   \end{xy}
\end{align} 
In this reaction, one unit of chemical {\em species} $A$ and one of $B$ react 
(at reaction rate $\kappa$) to form three units of $A$ and one of $C$.  
The {\em educt} (or {\em reactant} or {\em source}) $A+B$ and the {\em product} $3A+C$ are called {\em complexes}. 
We will refer to complexes such as $A+B$ that are the educt of a
  reaction as \emph{educt complexes}.
The concentrations of the three species, denoted by $x_{A},$ $x_{B}$, and
$x_{C}$, 
will change in time as the reaction occurs.  Under the assumption of {\em
mass-action kinetics}, species $A$ and $B$ react at a rate proportional to the
product of their concentrations, where the proportionality constant is the rate
constant $\kappa$.  Noting that the reaction yields a net change of two units in
the amount of $A$, we obtain the first differential equation in the following
system:
\begin{align*}
\frac{d}{dt}x_{A}~&=~2\kappa x_{A}x_{B}~, \\
\frac{d}{dt}x_{B} ~&=~-\kappa x_{A}x_{B}~, \\
\frac{d}{dt} x_{C}~&=~\kappa x_{A}x_{B}~.
\end{align*}
The other two equations arise similarly.  A {\em chemical reaction network}
consists of finitely many reactions.  
The differential equations that a network defines are comprised of a 
sum of the monomial contribution from the reactant of each 
chemical reaction in the network; these differential equations 
will be defined in equation~\eqref{CRN}.

\subsection{Chemical reaction systems} \label{sec:systems}
We now provide precise definitions.  A {\em chemical reaction network} is a finite directed graph 
whose vertices are labeled by complexes and whose edges are labeled by parameters (reaction
rate constants). 
Specifically, the digraph is denoted $G = (V,\Edg)$, with vertex set $V = \{1,2,\ldots,m\}$
and edge set $\,\Edg \subseteq \{(i,j) \in V \times V : \,i\not= j \}$.  
Throughout this article, the integer unknowns~$m$,~$s$, and~$r$ denote the numbers of
complexes, species, and edges (reactions), respectively.  
 {\em Linkage classes} refer to the connected components of a network, 
	and {\em terminal strong linkage classes} refer to
	the maximal strongly connected subgraphs in which there are no edges (reactions) from a complex in the subgraph to a complex outside the subgraph.
The vertex $i$ of $G$ represents the $i$-th chemical complex, and we associate to it the monomial 
$$ x^{y_i} \,\,\, = \,\,\, x_1^{y_{i1}} x_2^{y_{i2}} \cdots  x_s^{y_{is}}~. $$
	More precisely, if the $i$-th complex is $y_{i1} A + y_{i2} B + \cdots$ (where $y_{ij} \in \mathbb{Z}_{\geq 0}$ for $j=1,2,\dots,s$), then it defines the monomial $x_A^{y_{i1}} x_B^{y_{i2}} \cdots$.
For example, the two complexes in the network~\eqref{reaction_ex} give rise to 
the monomials $x_{A}x_{B}$ and $x^3_A x_C$, which determine two vectors 
$y_1=(1,1,0)$ and $y_2=(3,0,1)$.  
These vectors define the rows of an $m \times s$-matrix of non-negative integers,
which we denote by $Y=(y_{ij})$.
Next, the unknowns $x_1,x_2,\ldots,x_s$ represent the
concentrations of the $s$ species in the network,
and we regard them as functions $x_i(t)$ of time $t$.
The monomial labels form the entries in the following vector: 
$$ \Psi(x) \quad = \quad \bigl( x^{y_1}, ~ x^{y_2} , ~ \ldots ~ ,  ~ x^{y_m} \bigr)^t~.$$

A directed edge $(i,j) \in \Edg$ represents a reaction from 
the $i$-th chemical complex to the $j$-th chemical complex.
Each edge is labeled by a  positive parameter $\kappa_{ij}$ which represents the rate constant of the
reaction.  In this article, we will treat the rate constants $\kappa_{ij}$ as
unknowns; we are interested in the family of dynamical systems that arise from a
given network as the rate constants $\kappa_{ij}$ vary.  

The main application of our results are chemical reaction networks under mass-action
kinetics. Therefore, even if the principal results in \S~\ref{sec:toric} hold for general polynomial
dynamical systems, we assume in what follows mass-action kinetics. 
We now explain how mass-action kinetics defines a dynamical system
from a chemical reaction network.  
Let $A_\kappa$ denote the negative of the {\em Laplacian} of
the chemical reaction network $G$. In other words, $A_\kappa$ is the $m \times m$-matrix
whose off-diagonal entries
are the $\kappa_{ij}$ and whose row sums are zero.  
Now we define the {\em \CSMatrix} of size $s \times m$ to be 
\begin{align} \label{eq:sysMat}
\Sys~:=~Y^t \cdot A_{\kappa}^t~.
\end{align} 
  The reaction network $G$ 
defines the following dynamical system:
\begin{equation}
\label{CRN}
\frac{dx}{dt} ~=~\left( \frac{dx_1}{dt} ,\frac{dx_2}{dt}  ,\dots , \frac{dx_s}{dt}   \right)^t  ~=~ \Sys \cdot \Psi(x)~.
 \end{equation}
We see that the
right-hand side of each differential equation $dx_l/dt$ is a polynomial in the
polynomial ring $\mathbb{R}[(\kappa_{ij})_{(i,j) \in E}, x_1, x_2,
\dots, x_s]$.  A {\em chemical reaction system} refers to the
dynamical system (\ref{CRN}) arising from a specific chemical reaction
network $G$ and a choice of rate parameters $(\kappa^*_{ij}) \in
\mathbb{R}^{r}_{>0}$ (recall that $r$ denotes the number of
reactions). 

\begin{example} \label{ex:n=1}
The following chemical reaction network is the 1-site phosphorylation system:
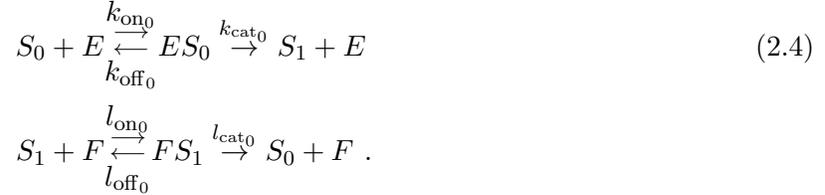
\begin{align} 
  S_0+E &
  \arrowschem{k_{\rm{on}_0}}{k_{\rm{off}_0}} ES_0
  \stackrel{k_{\rm{cat}_0}}{\rightarrow} S_1+E 
  \label{eq:net_phospho_1_complete} \\ 
  \notag
  S_1+F &
  \arrowschem{l_{\rm{on}_0}}{l_{\rm{off}_0}} FS_1
  \stackrel{l_{\rm{cat}_0}}{\rightarrow} S_0+ F ~.
\end{align}
The key players in this network
are a kinase enzyme ($E$), 
a phosphatase enzyme ($F$), and two substrates ($S_0$ and $S_1$).  
The substrate $S_1$ is obtained from the unphosphorylated protein $S_0$ by 
attaching a phosphate group to it via an enzymatic reaction involving $E$. 
Conversely, a reaction involving $F$ removes the phosphate group from $S_1$ 
to obtain $S_0$. The intermediate complexes~$ES_0$ and~$ES_1$ are the 
bound enzyme-substrate complexes.
Under the ordering of the 6 species as $(S_0,S_1,ES_0,FS_1,E,F)$ and the 6 
complexes as $(S_0+E,S_1+E,ES_0,S_0+F,S_1+F,FS_1)$, the matrices whose product
defines the dynamical system (\ref{CRN}) follow:
\begin{equation*}
\Psi(x) ~ = ~  \left(  x_{S_0}x_E ,~ x_{S_1}x_E ,~ x_{ES_0} ,~ x_{S_0}x_F ,~ x_{S_1}x_F ,~ x_{FS_1}  \right)^t	
	~=~ \left(  x_1 x_5,~ x_2 x_5,~ x_3,~ x_1 x_6,~ x_2 x_6,~ x_4 \right)^t, 
\end{equation*}
\begin{equation*}
  Y^t ~=~ \left[  
    \begin{array}{llllllllll}
      1 & 0 &  0 &  1 & 0  & 0  \\  
      0 & 1 &  0 &  0 & 1  & 0  \\  
      0 & 0 &  1 &  0 & 0  & 0  \\  
      0 & 0 &  0 &  0 & 0  & 1  \\  
      1 & 1 &  0 &  0 & 0  & 0  \\  
      0 & 0 &  0 &  1 & 1 & 0      
   \end{array}
  \right], {\rm ~ and}
\end{equation*}
\begin{equation*}
  A^t_{\kappa} ~:=~ 
  \left[  
    \begin{array}{cccccccccc}
      -k_{\rm{on}_0} & 0 & k_{\rm{off}_0}  & 0 & 0 & 0  \\					
      0 & 0 & k_{\rm{cat}_0} &  0 & 0 & 0 \\	
      k_{\rm{on}_0} & 0 & -k_{\rm{off}_0}-k_{\rm{cat}_0} & 0 & 0 & 0  \\	
      0 & 0 & 0 & 0 & 0 & l_{\rm{cat}_0} \\								
      0 & 0 & 0 & 0  & -l_{\rm{on}_0} & l_{\rm{off}_0}  \\		
      0 & 0 & 0 & 0 & l_{\rm{on}_0} & -l_{\rm{cat}_0}-l_{\rm{off}_0}  \\	
    \end{array}
  \right].
\end{equation*}
We will study generalizations of this network in this article.
\end{example}

The {\em stoichiometric subspace} is the vector subspace spanned by the {\em
reaction vectors} 
$y_j-y_i$ (where $(i,j)$ is an edge of $G$), and we will denote this space by
$\St$: 
\begin{equation*}
\St~:=~ \mathbb{R} \{ y_j-y_i~|~ (i,j) \in \Edg \}~.
\end{equation*}
In the earlier example shown in~\eqref{reaction_ex}, we have $y_2-y_1 =(2,-1,1)$, which
means 
that with the occurrence of each reaction, two units of $A$ and one of $C$ are
produced, while one unit of $B$ is consumed.  This vector $(2,-1,1)$ spans the
stoichiometric subspace $\St$ for the network~\eqref{reaction_ex}. 
Note that the  vector $\frac{dx}{dt}$ in  (\ref{CRN}) lies in
$\St$ for all time $t$.   
In fact, a trajectory $x(t)$ beginning at a positive vector $x(0)=x^0 \in
\R^s_{>0}$ remains in the {\em stoichiometric compatibility class} (also called an 
``invariant polyhedron''), which we
denote by
\begin{align}\label{eqn:invtPoly}
\invtPoly~:=~(x^0+\St) \cap \mathbb{R}^s_{\geq 0}~, 
\end{align}
for all positive time.  In other words, this set is forward-invariant with
respect to the dynamics~(\ref{CRN}).    
It follows that any stoichiometric compatibility class of a network has the same
dimension as the stoichiometric subspace. 

\subsection{Steady states} \label{sec:steadyStes}

We present the definition of systems with toric steady states.
For background information on the algebraic tools we use, we refer the reader
to the nice textbook of Cox, Little, and O'Shea \cite{CLO}.

Recall that an ideal in $\R[x_1,x_2,\dots,x_s]$ is called a {\em binomial ideal} if it can
be generated by {\em binomials} (i.e., polynomials with at most two terms). The basic building blocks of binomial 
ideals are the {\em prime} binomial ideals, which are called {\em toric ideals}~\cite{es96}.

\begin{definition}\label{def:ss}
Consider a polynomial dynamical system $dx_i/dt =
    f_i(x),$ for $i = 1,2, \dots, s,$ with $f_1, f_2, \dots, f_s \in \R[x_1, x_2,
    \dots, x_s]$. We are interested in the {\em real} zeros of the
    {\em steady state ideal}:
    \begin{displaymath}
      J_{\Sys \Psi}~=~\langle f_1,f_2, \dots, f_s \rangle 
      = \left\{ \sum_{i=1}^s h_i(x) f_i(x) \quad | \quad  h_i(x) \in
        \R[x_1, x_2, \dots, x_s] ~{\rm for }~  1 \leq i \leq s \right\}.
    \end{displaymath}
    The real zeros of $J_{\Sys \Psi}$ are called {\em steady states},
    and the term {\em steady state locus} is used to denote 
	the	set of    real zeros of $J_{\Sys \Psi}$:
    \begin{displaymath}
      \left\{ x^* \in \mathbb{R}^s \quad | \quad  f_1(x^*) = f_2(x^*) = \cdots =f_s(x^*) =0
      \right\}.
    \end{displaymath}
    We say that the polynomial dynamical system has {\em toric steady
      states} if $J_{\Sys \Psi}$ is a binomial ideal and it admits real zeros.
\end{definition}
\noindent

We are interested in {\em positive steady states} 
$x\in \mathbb{R}^s_{> 0}$ and will not be concerned with {\em boundary steady
states} 
$x\in  \left( \mathbb{R}^s_{\geq 0} \setminus \mathbb{R}^s_{> 0} \right)$.

This article focuses on mass-action kinetics \CRSsNoSpace.  
In this case, the polynomials $f_1,f_2, \dots, f_s$ correspond to the rows of the
system~\eqref{CRN}. 
In general, having toric steady states depends both on the reaction
network and on the particular rate constants, as the following simple
example shows.
\begin{example}[Triangle network] \label{ex:toric_depends_on_rates}
 Let $s=2$, $m = 3$, and  let $G$ be the following network:
 
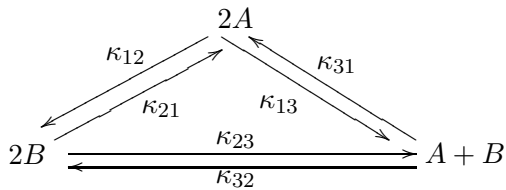
\[
 \begin{xy}<12mm,0cm>:
(0,1.3)                  ="A+B"  *+!D{2A}  *{}; 
 (2,0)                  ="C"  *+!L{A+B}  *{}; 
(-2,0)                  ="D"  *+!R{2B}  *{}; 
 (0.8,0.8)                  =""  *+!DL{{\kappa_{31} } }  *{};
(0.8,0.55)                  =""  *+!R{{ \kappa_{13} }}  *{};
(0,-.085)                  =""  *+!U{{ \kappa_{32}}}  *{};
(0,-.03)                  =""  *+!D{{ \kappa_{23}}}  *{};
(-0.5,0.5)                  =""  *+!R{{\kappa_{21}}}  *{};
(-.9,.9)                  =""  *+!DR{{\kappa_{12}}}  *{};
   {\ar "A+B"+(-0.15,0)*{};"C"+(-0.3,.15)*{}  }; 
   {\ar "C"+(0,0.15)*{};"A+B"+(0.15,0)*{}};       
   {\ar "A+B"+(-0.30,0)*{};"D"+(-0.15,.3)*{}  }; 
   {\ar "D"+(0,.15)*{};"A+B"+(-.15,-0.15)*{}}; 
   {\ar "C"+(0,-.15)*{};"D"+(0.15,-.15)*{}}; 
   {\ar "D"+(0.15,0)*{};"C"+(0,0)*{}}; 
  \end{xy}
  \]
 We label the three complexes as 
  $x^{y_1}= x_1^2$, 
  $x^{y_2} = x_2^2$,
  $x^{y_3} = x_1 x_2$, and we define
  $\kappa_{ij}$ to be 
the (real positive) rate constant of the reaction from complex $x^{y_i}$ to complex
$x^{y_j}$.  The resulting mass-action kinetics
system \eqref{CRN}
equals
\begin{equation*}
\frac{d x_1}{dt} \quad = \quad - \, \frac{d x_2}{dt} \quad  = \quad  (-2
\kappa_{12}-\kappa_{13}) x_1^2 + (2 \kappa_{21}+\kappa_{23}) x_2^2 
	+ (\kappa_{31}-\kappa_{32}) x_1 x_2~.
  \end{equation*}
Then, the steady state locus in $\mathbb{R}^2$ is defined by this single trinomial. 
As only the coefficient of $x_1x_2$ can be zero, this
system has toric steady states if and only if 
$\kappa_{31} = \kappa_{32}$.
\end{example}

A \CRS exhibits {\em multistationarity} if there exists a stoichiometric
compatibility class 
$\invtPoly$ with two or more steady states in its relative interior. A system
may admit multistationarity for all, some, or no choices of
positive rate constants $\kappa_{ij}$; if such rate constants exist, then we say that 
the network {\em has the capacity for multistationarity}.  

\subsection{The deficiency of a \CRN} \label{subsec:deficiency}
The  {\em deficiency} $\delta$ of a chemical reaction network is an important invariant.   
For a chemical reaction network, recall that $m$ denotes the number
of complexes. Denote by  $l$ 
the number of linkage classes. Most of the networks considered in this article
have the property that each linkage class contains a unique terminal strong linkage class.
In this case, Feinberg showed that the deficiency of the network can be computed in the following way: 
$$ \delta~:=~m-l-\dim(\St)~,$$
where $\St$ denotes the stoichiometric subspace.  
Note that in this case the deficiency depends only on the
reaction network and not on the specific values of the rate constants.
The deficiency of a reaction network is non-negative
because it can be interpreted as the dimension of a
certain linear subspace \cite{Feinberg72} 
or the codimension of a certain ideal \cite{TDS}.
For systems arising from 
zero-deficiency networks and networks whose linkage classes have 
deficiencies zero or one, there are many results due to Feinberg that concern the existence,
uniqueness, and stability of steady states \cite{Feinberg72,Feinberg89,fe95,Fein95DefOne}.

\section{Sufficient conditions for the existence of toric steady states}\label{sec:toric}

The main results of this section, Theorems~\ref{th:binomial_ideal},~\ref{th:toric},
and~\ref{thm:necessary}, give sufficient conditions for a \CRS to have toric
steady states and state criteria for these systems to have  positive toric steady states.
Theorems~\ref{thm:toric_param} and~\ref{thm:toricgeneral_param}
give a monomial parametrization of the steady state locus in this case.

We first state several conditions and intermediate results that will lead to
Theorem~\ref{th:toric}.
 Recall that a {\em partition} of $\{1,2,
  \dots, m\}$ is a collection of nonempty disjoint subsets $I_1, I_2,
  \dots, I_d$ with respective cardinalities $l_1, l_2, \dots, l_d$
  such that their union equals $\{1,2, \dots, m\}$ (or equivalently,
  such that $l_1 + l_2 + \dots + l_d =m$). The support ${\rm supp(}b{\rm )}$
of a real vector $b\in\R^m$ is the subset of indices corresponding to
the nonzero entries of $b$.  The following condition requires that a certain
linear subspace has a basis with disjoint supports.

\begin{condition} \label{cond:1}
  For a \CRS given by a network $G$ with $m$ complexes and reaction rate constants
  $\kappa^*_{ij}$, 
  let $\Sys$ denote its \CSMatrix~\eqref{eq:sysMat}, and set
  $d:=\dim( \ker(\Sys) )$.    
    We say that the \CRS satisfies Condition~\ref{cond:1}, if there exists
    a partition $I_1,I_2, \dots, I_d$ of $\{1,2, \dots, m\}$ and a
    basis $b^1, b^2, \ldots, b^d \in \R^m$ of $\ker(\Sys)$ with ${\rm
      supp}(b^i) = I_i$.
\end{condition}
\begin{remark}
  Conditions~\ref{cond:1},~\ref{cond:2}, and~\ref{cond:3} in this
  article are essentially linear algebra conditions. When we consider
  a specific choice of rate constants $\kappa^*_{ij}$, checking these
  conditions involves computations over $\R$. However, the objects of
  interest (such as the subspace in Condition~\ref{cond:1}) are
  parametrized by the unknown rate constants $\kappa_{ij}$, so
  verifying the conditions can become quite complicated for large
  networks. In this case, we need to do linear computations over the
  field $\Q(k_{ij})$ of rational functions on these parameters and
  check semialgebraic conditions on the rate constants
  (cf. Remark~\ref{rmk:condDelta}).
\end{remark}

Condition~\ref{cond:1} implies that the steady state ideal $J_{\Sys
  \Psi}$ is binomial: 

  \begin{theorem}
    \label{th:binomial_ideal}
    Consider a \CRS with $m$ complexes, and let $d$ denote the
    dimension of $\ker(\Sys)$.  Assume that Condition~\ref{cond:1}
    holds (i.e., there exists
    a partition $I_1,I_2, \dots, I_d$ of $\{1,2, \dots, m\}$ and a
    basis $b^1, b^2, \ldots, b^d \in \R^m$ of $\ker(\Sys)$ with ${\rm
      supp}(b^i) = I_i$). Then 
    the steady state ideal $J_{\Sys\Psi}$ is
    generated by the binomials 
    \begin{align} 
      \label{eq:bins_from_basis}
      b^j_{j_1}x^{y_{j_2}}-b^j_{j_2} x^{y_{j_1}} \text{, for
        all $j_1$, $j_2 \in I_j$, and for all  $1 \leq j \leq d$.}
    \end{align}
  \end{theorem}

  \begin{proof}
Consider the vectors $\beta^j_{j_1,j_2}=b^j_{j_1}e_{j_2}-b^j_{j_2}e_{j_1} \in \R^m$ 
for all $j_1$, $j_2 \in I_j$, for all $1 \leq j \leq d$. It is straightforward 
to check that these vectors span the orthogonal complement $\ker(\Sys)^\bot$ of 
the kernel of $\Sys$. But by definition, this complement is spanned by the 
rows of the matrix $\Sys$. Therefore, the binomials 
$b^j_{j_1}\Psi_{j_2}(x)-b^j_{j_2}\Psi_{j_1}(x)$ are $\R$-linear combinations of 
the polynomials $f_1(x), f_2(x), \dots, f_s(x)$, and vice-versa. And so the binomials 
in~\eqref{eq:bins_from_basis} give another system of generators of $J_{\Sys\Psi}$. 
\end{proof}

Note that Theorem~\ref{th:binomial_ideal} does not provide any
information about the existence of (toric) steady states (i.e.\ {\em
    real} solutions to the binomials~\eqref{eq:bins_from_basis},
  cf.\ Definition~\ref{def:ss}),
let alone {\em positive steady states}. In general, this
  is a question of whether a parametrized family of polynomial
  systems has
  real solutions.
For this purpose two further conditions are needed:

  \begin{condition} \label{cond:2}
    Consider a \CRS given by a network $G$ with $m$ complexes and reaction rate constants
  $\kappa^*_{ij}$ that satisfies
    Condition~\ref{cond:1} for the partition $I_1,I_2,
    \dots, I_d$ of $\{1,2, \dots, m\}$ and a basis $b^1, b^2,
    \ldots, b^d \in \R^m$ of $\ker(\Sys)$ (with ${\rm supp}(b^i) =
    I_i$).
    We say that this \CRS \textbf{additionally} satisfies
    Condition~\ref{cond:2}, if for all $j \in \{1,2, \dots, d\}$, the
    nonzero entries of $b^j$ have the same sign, that is, if 
    \begin{equation}
      \label{eq:sign_equality}
      \sign\left(b^j_{j_1}\right) = \sign\left(b^j_{j_2}\right)\text{,
        for all $j_1$, $j_2 \in I_j$, for all $1 \leq j \leq d$.}
    \end{equation}
  \end{condition}

  The next result can be used to check the validity of Condition~\ref{cond:2}.
  \begin{lemma}
    \label{lem:coefficient_signs}
    Consider a \CRS with $m$ complexes that satisfies
    Condition~\ref{cond:1} for the partition $I_1,I_2,
    \dots, I_d$ of $\{1,2, \dots, m\}$ and the basis $b^1, b^2,
    \ldots, b^d \in \R^m$ of $\ker(\Sys)$.
    Let $j \in \{1,2,\dots,d\}$, There exists an $(l_j-1)\times
    l_j$ submatrix $\Sys_j$ of $\Sys$ with columns indexed by the elements
    of $I_j$ and linearly independent rows (that is, $rank(\Sys_j)=l_j-1$). 
    Let $\Sys_j$ be any such matrix. For $i \in \{1, \dots, l_j\}$, call 
    $\Sys_j(i)$ the submatrix of $\Sys_j$ obtained by deleting its $i$-th column.
    Then the system satisfies
   Condition~\ref{cond:2} (that is, equations~\eqref{eq:sign_equality} are satisfied)
       if and only if, for all $j \in \{1,2, \dots, d\}$, the sign of
    $\det(\Sys_j(i))$ is different from the sign of $\det(\Sys_j(i+1))$ for 
    $1\leq i \leq l_j-1$.
  \end{lemma}
  \begin{proof}
   First, note that the kernel of the submatrix of $\Sys$ formed by the 
   columns indexed by $I_j$ has dimension one and is spanned by 
   the vector $b'_j$ which consists of the $l_j$ entries of $b^j$ that
   are indexed by $I_j$. 
   So there exist $l_j-1$ rows that give a matrix $\Sys_j$ as in the statement.

   By a basic result from Linear Algebra, the kernel of $\Sys_j$ is 
   spanned by the vector $v'$ with $i$-th entry equal to $(-1)^i \det(\Sys_j(i))$.
   As the vector $b'_j$ must be a multiple 
   of $v'$, it is immediate that~\eqref{eq:sign_equality} holds 
   if and only if the sign of $\det(\Sys_j(i))$ is different from the sign of 
   $\det(\Sys_j(i+1))$ for $1\leq i \leq l_j-1$.
  \end{proof}

  Condition~\ref{cond:2}
   is necessary for the existence of
  positive real solutions to the system defined by setting the 
  binomials~(\ref{eq:bins_from_basis}) to zero. In working
  towards sufficiency, observe that the system can be rewritten as
  \begin{displaymath}
    x^{y_{j_1}-y_{j_2}} ~=~ \frac{b^j_{j_1}}{b^j_{j_2}}\text{,
        for all $j_1$, $j_2 \in I_j$ and for all $1 \leq j \leq d$.}
  \end{displaymath}
  Note that Condition~\ref{cond:2} implies that 
  the right-hand side of the
  above equation is positive. In addition, we are interested in positive
  solutions $x\in \Rplus^s$, so we now apply $\ln\left(\cdot\right)$ 
  to both sides and examine the
  solvability of the resulting \emph{linear system}:
  \begin{displaymath}
    \ln x\, \left(y_{j_1}-y_{j_2}\right)^t ~=~ \ln
    \frac{b^j_{j_1}}{b^j_{j_2}}\text{, for all $j_1$, $j_2 \in I_j$
    and 
      for all $1 \leq j \leq d$,}
  \end{displaymath}
where   $\ln x =(\ln(x_1),\ln(x_2),\dots,\ln(x_s))$.
  Now collect the differences $(y_{j_1}-y_{j_2})^t$ as columns of a
  matrix 
  \begin{equation}
    \label{eq:def_Delta}
    \Delta := \left[
      \left(y_{j_1}-y_{j_2}\right)^t
    \right]_{\forall j_1,\, j_2 \in I_j,\, \forall 1 \leq j \leq d}~,
  \end{equation}
  and define the (row) vector
  \begin{equation}
    \label{eq:def_Theta}
    \Theta_\kappa := \left(
      \ln \frac{b^j_{j_1}}{b^j_{j_2}} 
    \right)_{\forall j_1,\, j_2 \in I_j,\, \forall 1 \leq j \leq d}\ .
  \end{equation}
  Observe that the basis vectors $b^j$ and hence the vector $\Theta_\kappa$ 
  depend on the
  rate constants. The binomials (\ref{eq:bins_from_basis}) admit a real
  positive solution (in the presence of Condition~\ref{cond:2}), if and only if the linear system
  \begin{equation}
    \label{eq:LIN}
    (\ln x)\, \Delta = \Theta_\kappa
  \end{equation}
  has a real solution $(\ln x) \in \mathbb{R}^s$. 
  This is the motivation for our final condition and
  Theorem~\ref{th:toric} below:

  \begin{condition} \label{cond:3}
    Consider a \CRS given by a network $G$ with $m$ complexes and reaction rate constants
  $\kappa^*_{ij}$  that satisfies both
    Condition~\ref{cond:1} (i.e.\ there exists a partition $I_1,I_2,
    \dots, I_d$ of $\{1,2, \dots, m\}$ and a basis $b^1, b^2, \ldots,
    b^d \in \R^m$ of $\ker(\Sys)$ with ${\rm supp}(b^i) = I_i$) and
    Condition~\ref{cond:2} (i.e., the coefficients of each binomial
    in equation~(\ref{eq:bins_from_basis})
    are of the same sign). Recall the matrix~$\Delta$ and the vector
    $\Theta_\kappa$ (defined in equations~(\ref{eq:def_Delta})
     and~(\ref{eq:def_Theta}), respectively). Let $U$ be a matrix with integer entries
    whose columns form a basis of the kernel of $\Delta$, that is, $U$ is an integer 
    matrix of maximum column rank such that the following matrix product is a zero matrix
    with $s$ rows:
    \begin{displaymath}
      \Delta\, U ~=~ \bf{0}\ .
    \end{displaymath}
   We say that this \CRS  \textbf{additionally} satisfies
    Condition~\ref{cond:3} 
    if the linear system~\eqref{eq:LIN} has a real solution 
    $(\ln x) \in \mathbb{R}^s$.
    Equivalently, the Fundamental Theorem 
    of Linear Algebra \cite{Strang} implies that
    equation \eqref{eq:LIN} has a solution, if and only if
     \begin{equation}
       \label{eq:kappa_cond}
       \Theta_\kappa\, U = 0\ .
     \end{equation}
   \end{condition}
   
  \begin{remark}\label{rmk:condDelta}
    Conditions \ref{cond:2} and \ref{cond:3} impose semialgebraic constraints
   on the rate constants:
    \begin{itemize}
    \item If the matrix $\Delta$ defined in (\ref{eq:def_Delta}) has
      full column rank (i.e.\ the right kernel is
        trivial), then $U$ is the zero vector. 
      It follows that
      equation~(\ref{eq:kappa_cond}) holds, and hence,  
      Condition~\ref{cond:3}
      is trivially satisfied for any positive vector of rate
      constants.  We will see that this is the case for multisite
      phosphorylation networks.
    \item If the matrix $\Delta$ does not have full column
        rank (i.e.\ there exists a nontrivial right kernel),
      then equation~(\ref{eq:kappa_cond})  can be translated to a system of
      polynomial equations in the rate constants.
    \end{itemize}
  \end{remark}

  Now we can state sufficient conditions for a \CRS to admit
  positive toric steady states:
  \begin{theorem}[Existence of positive toric steady states]
    \label{th:toric}
    Consider a \CRS with $m$ complexes which satisfies 
    Condition~\ref{cond:1} and hence has a binomial steady state ideal
      $J_{\Sys \Psi}$. Then this \CRS admits
    a positive toric steady state if and only if 
    Conditions~\ref{cond:2} and~\ref{cond:3} hold. 
  \end{theorem}
  \begin{proof}
    Assume that Conditions \ref{cond:1},~\ref{cond:2}, and~\ref{cond:3}
    hold. 
Lemma~\ref{lem:coefficient_signs} implies that the coefficients of
    the binomial system are of the same sign, hence $\Delta$ and
    $\Theta_\kappa$ given in equations~(\ref{eq:def_Delta}) and
    (\ref{eq:def_Theta}) and the linear system (\ref{eq:LIN})
    are well-defined.
	Then Condition~\ref{cond:3} gives a solution $(\ln x)$ to the 
	system (\ref{eq:LIN}), which immediately gives a positive steady state
	$x\in\Rplus^s$ of the \CRSNoSpace.

    On the other hand, assume that Condition \ref{cond:1} holds and that
    the system admits a positive steady state,
    that is, the binomial system~(\ref{eq:bins_from_basis}) has a
    positive real solution. In this case the coefficients of the
    binomials must be of the same sign, which implies that Condition
    \ref{cond:2} holds additionally. Again, positive real solutions of
    the binomial system imply solvability of the linear system
    (\ref{eq:LIN}) and thus, Condition~\ref{cond:3} is satisfied as well.
  \end{proof}
 
\begin{remark}[Existence of steady states using fixed point arguments]
  \label{rem:fixed-point-arguments}
  In some cases, one can establish the existence of positive steady states by using fixed-point arguments.  If the stoichiometric compatibility classes of a network are bounded, 
a version of the Brouwer fixed point theorem guarantees that a non-negative steady state exists in each compatibility class.  If moreover the \CRS has no boundary steady states, 
we deduce the existence of a positive steady state in each compatibility class.
For example, the multisite phosphorylation networks that are studied in this article have this property. 
The positive conservation laws in~\eqref{eqn:conservation} ensure boundedness and Lemma~\ref{lem:nobss}
shows that no boundary steady states can occur.

The focus of our results, however, is slightly different. We are more
 interested in parametrizing the steady state locus (and hence all positive
 steady states) and less with the actual number of steady states within a
 given stoichiometric compatibility class (apart from
 Section~\ref{sec:multistat}, where we are concerned with
 compatibility classes having at least two distinct positive steady states).
 Moreover, using fixed point arguments, the existence of positive steady
 states may only be deduced if the \CRS has no boundary steady states, which
 is somewhat rare in examples from Computational Biology. Our results do not require any
 information about boundary steady states.
\end{remark}

\begin{example}[Triangle network, continued] \label{ex:return_triangle}
We return to Example~\ref{ex:toric_depends_on_rates}
to illustrate the three conditions. First, $\ker(\Sys)$ is the plane in $\R^3$
orthogonal to the vector 
$( -2\kappa_{12}-\kappa_{13}, 2\kappa_{21} +  \kappa_{23} ,  \kappa_{31}- \kappa_{32})$. 
 It follows that the partition $\{1,2\},\{3\}$ works 
to satisfy Condition~\ref{cond:1} if and only if $\kappa_{31}=\kappa_{32}$.  
Therefore, for a \CRS arising from the Triangle network, 
Condition~\ref{cond:1} holds (with partition $\{1,2\},\{3\}$)
if and only if the system has toric steady states.
The forward direction is an application of Theorem~\ref{th:binomial_ideal},
while for general networks the reverse implication is false:
we will see in Example~\ref{ex:ShiFein}
that there are networks with toric steady states that do not satisfy
Condition~\ref{cond:1} for any partition.

Next, for those systems for which $\kappa_{31}=\kappa_{32}$, 
Condition~\ref{cond:2} comes down to verifying that the entries of the
vector $(-2\kappa_{12}-\kappa_{13}, 2\kappa_{21}+\kappa_{23})$ 
have opposite signs, which is clearly true for positive rate constants.
Finally, Condition~\ref{cond:3} asks (again, in the 
$\kappa_{31}=\kappa_{32}$ setting) whether the following linear system
has a real solution $\left(\ln x_1,~\ln x_2\right) \in \mathbb{R}^2$:
\[	
\left(\ln x_1,~\ln x_2\right) 
\underbrace{
  \left(
    \begin{array}{r}
      2 \\ -2
    \end{array}
  \right)
}_{=\Delta}  ~=~ \underbrace{
  \ln \left( \frac{2\kappa_{21}+\kappa_{23}}{2\kappa_{12}+\kappa_{13}}
      \right)
}_{=\Theta_\kappa}~,
\]
which is clearly true. This linear equation
 arises from the binomial equation
\[	\left( 2\kappa_{12}+\kappa_{13} \right) x_1^2 -
		\left( 2\kappa_{21}+\kappa_{23} \right) x_2^2 ~=~0~.
\]
As Condition~\ref{cond:3} holds,
Theorem~\ref{th:toric} implies that these systems admit
positive steady states.
\end{example}

Under the hypothesis of Theorem~\ref{th:toric}, the following result shows how 
to parametrize the steady state locus.

  \begin{theorem}\label{thm:toric_param}
    Consider a \CRS that satisfies Conditions~\ref{cond:1}, \ref{cond:2}, and
    \ref{cond:3}. Let $A \in \mathbb{Z}^{w \times s}$ be a matrix
    of maximal rank $w$ such that $\ker(A)$ equals the span of all the
    differences $y_{j_2}-y_{j_1}$ for $j_1, j_2 \in I_j$, where $1\leq j
    \leq d$. For $1 \leq i \leq s$, we let $A_i$  denote the $i$-th
    column of $A$. Let $\tilde{x}\in \Rplus^s$ be a positive steady
    state of the chemical reaction system. Then all positive solutions
    $x\in\Rplus^s$ to the binomial system (\ref{eq:bins_from_basis}) can be
    written as
    \begin{equation}
      \label{eq:para_steady_states}
      x  ~=~ \left(\tilde{x}_1 \, t^{A_1},~ \tilde{x}_2 \, t^{A_2}, ~
        \dots, ~ \tilde{x}_s \, t^{A_s}\right),
    \end{equation}
    for some $t \in \mathbb{R}_{>0}^{w}$ (where we are using the
    standard notation for multinomial exponents). In particular, the
    positive steady state locus has dimension $w$
    and can be parametrized by monomials in the concentrations.
    Any two distinct positive steady states $x^1$ and $x^2$ satisfy
    \begin{equation}
      \label{eq:ln_lin}
      \ln x^2 - \ln x^1 \in \im\left(A^t\right) ~=~ {\rm span}\left\{
        y_{j_2}-y_{j_1}\, | \,  j_1, j_2 \in I_j, 1\leq j
    \leq d \right\}^{\perp}.
    \end{equation}
  \end{theorem}

  \begin{proof}
    By definition, the rows of $A$ span the orthogonal complement of the
    linear subspace spanned by the differences $y_{j_2}-y_{j_1}$ for
    $j_1, j_2 \in I_j$, $1\leq j \leq d$. Let $\tilde{x}\in \Rplus^s$
    be a positive steady state of the chemical reaction system; in
    other words, it is a particular positive solution for the
    following system of equations:
    \begin{displaymath}
      b^j_{j_1}x^{y_{j_2}}-b^j_{j_2}x^{y_{j_1}}~=~0 \quad \text{ for all
      } \, j_1, j_2 \in I_j, \, \text{ and for all } \, 1 \leq j \leq
      d\ .
    \end{displaymath}
    (Here the $b^j$ are the basis vectors of $\ker(\Sys)$ with disjoint
    support.) Then it follows from basic results on binomial equations
    that all positive solutions $x\in\Rplus^s$ to the above system of binomial
    equations can be written as 
    \[ x \quad =
    \quad \left(\tilde{x}_1 \, t^{A_1},~ \tilde{x}_2 \, t^{A_2}, ~
      \dots, ~ \tilde{x}_s \, t^{A_s}\right), \]
    for some $t \in \mathbb{R}_{>0}^{w}$. In particular, the positive 
    steady state locus has $w$ degrees of freedom. \\
    For the convenience of the reader, we expand now the previous
    argument. In fact, it is easy to check that any vector of this shape
    is a positive solution.  We first let $x^*$ be a particular positive
    solution of the above binomials. Then
    $\frac{x^*}{\tilde{x}}~:=~\left( \frac{x^*_1}{\tilde{x_1}},
      \frac{x^*_2}{\tilde{x_2}},\dots \frac{x^*_s}{\tilde{x_s}}
    \right)$ is a positive solution of the system of equations:
    \begin{displaymath}
      x^{y_{j_2}}-x^{y_{j_1}}=0 \quad \text{ for all } j_1, j_2 \in I_j,
      \, \text{ for all } 1 \leq j \leq d~.
    \end{displaymath}
    Therefore,
    $\left(\frac{x^*}{\tilde{x}}\right)^{y_{j_2}-y_{j_1}}=1$. Or,
    equivalently, $\ln \left(\frac{x^*}{\tilde{x}}\right) \cdot
    (y_{j_2}-y_{j_1})=0$. This implies that $\ln
    \left(\frac{x^*}{\tilde{x}}\right)$ belongs to the rowspan of $A$,
    and this means there exist $\lambda_1, \lambda_2, 
    \dots, \lambda_{w}$ such
    that, if $\mathcal{A}_1, \mathcal{A}_2, \dots, \mathcal{A}_{w}$ represent the rows of $A$, then we
    can write $$\left(\ln
      \left(\frac{x^*}{\tilde{x}}\right)\right)_i=\lambda_1 (\mathcal{A}_1)_i + \lambda_2 (\mathcal{A}_2)_i +
    \dots  +\lambda_{w} (\mathcal{A}_{w})_i~, \; \, \text{ for all } \, 1 \leq i
    \leq s~.$$ If we call $t_\ell:=\exp(\lambda_\ell)$ for $1 \leq \ell
    \leq w$,  then $x^*_i=\tilde{x}_i t^{A_i}$ for all $1 \leq i \leq
    s$, which is what we wanted to prove.
  \end{proof}

We now turn to the case of a network 
for which Condition~\ref{cond:1} holds with the {\em same} partition
for all choices of rate constants.  The following result, which follows
immediately from Theorem~\ref{th:toric}, states that for 
such a network, the semialgebraic set of rate constants that give rise to 
systems admitting positive steady states is defined by 
Conditions~\ref{cond:2} and~\ref{cond:3}.

\begin{corollary} \label{cor:toric_network_same_partition}
Let $G$ be a \CRN with $m$ complexes and $r$ reactions, and assume that there
	exists a partition $I_1,I_2 \dots, I_d$ of the $m$ complexes  
	such that for any choice of reaction rate constants, the resulting
	\CRS satisfies Condition~\ref{cond:1} with this partition.  
	Then a vector of reaction rate constants $\kappa_{ij}^* \in \Rplus^r$ 
	gives rise to a system that admits a positive steady state if and only if 
	$\kappa_{ij}^*$ satisfies Conditions~\ref{cond:2} and~\ref{cond:3}.
\end{corollary}

In the following example, we see that the 2-site phosphorylation network
satisfies the hypothesis of Corollary~\ref{cor:toric_network_same_partition}.
The 2-site system generalizes the 1-site system in Example~\ref{ex:n=1}, 
and we will consider general $n$-site systems in Section~\ref{sec:multisite_toric}.

\begin{example}[2-site phosphorylation system] \label{ex:trunning}
 The dual phosphorylation network arises from the 1-site 
 network~\eqref{eq:net_phospho_1_complete} by allowing a total of two phosphate
 groups to be added to the substrate of $S_0$ rather than  only one.
Again there are two
enzymes ($E$ and $F$), but now there are $3$ substrates ($S_0$, $S_1,$ and
$S_2$).  The substrate $S_i$ is the substrate obtained from $S_0$ by attaching
$i$ phosphate groups to it.  Each substrate can accept (via an enzymatic
reaction involving $E$) or lose (via a reaction involving $F$) at most one
phosphate; this means that the mechanism is ``distributive''.  In addition, we
say that the phosphorylation is ``sequential'' because multiple phosphate groups
must be added in a specific order, and removed in a specific order as well.
 \begin{align} 
\nonumber 
S_0+E &
\arrowschem{k_{\rm{on}_0}}{k_{\rm{off}_0}} ES_0
\stackrel{k_{\rm{cat}_0}}{\rightarrow} S_1+E 
\arrowschem{k_{\rm{on}_{1}}}{k_{\rm{off}_{1}}}ES_{1}
\stackrel{k_{\rm{cat}_{1}}}{\rightarrow} S_2+ E \\
\label{eq:net_phospho_2_complete} \\ 
\nonumber 
S_2+F &
\arrowschem{l_{\rm{on}_{1}}}{l_{\rm{off}_{1}}} FS_2
\stackrel{l_{\rm{cat}_{1}}}{\rightarrow} S_1+F 
\arrowschem{l_{\rm{on}_0}}{l_{\rm{off}_0}} FS_1
\stackrel{l_{\rm{cat}_0}}{\rightarrow} S_0+ F
\end{align}
We order the 9 species as $(S_0,S_1,S_2,ES_0,ES_1,FS_1,FS_2,E,F)$, 
and we order the 10 complexes as 
$(S_0+E,S_1+E,S_2+E,ES_0,ES_1,S_0+F,S_1+F,S_2+F,FS_1,FS_2)$.  
The $9\times 10$-matrix $Y^t$ and the $10\times10$-matrix $A^t_\kappa$ for this system 
are the following:
\[
  Y^t = \left[  
    \begin{array}{llllllllll}
      1 & 0 & 0 & 0 & 0 & 1 & 0 & 0 & 0 & 0 \\  
      0 & 1 & 0 & 0 & 0 & 0 & 1 & 0 & 0 & 0 \\  
      0 & 0 & 1 & 0 & 0 & 0 & 0 & 1 & 0 & 0 \\  
      0 & 0 & 0 & 1 & 0 & 0 & 0 & 0 & 0 & 0 \\  
      0 & 0 & 0 & 0 & 1 & 0 & 0 & 0 & 0 & 0 \\  
      0 & 0 & 0 & 0 & 0 & 0 & 0 & 0 & 1 & 0 \\  
      0 & 0 & 0 & 0 & 0 & 0 & 0 & 0 & 0 & 1 \\  
      1 & 1 & 1 & 0 & 0 & 0 & 0 & 0 & 0 & 0 \\  
      0 & 0 & 0 & 0 & 0 & 1 & 1 & 1 & 0 & 0     
   \end{array}
  \right]\]
\begin{small} 
\[
  A^t_{\kappa} := 
  \left[  
    \begin{array}{cccccccccc}
      -k_{\rm{on}_0} & 0 & 0 & k_{\rm{off}_0} & 0 & 0 & 0 & 0 & 0 & 0 \\
      0 & -k_{\rm{on}_1} & 0 & k_{\rm{cat}_0} & k_{\rm{off}_1} & 0 & 0 & 0 & 0 & 0 \\
      0 & 0 & 0 & 0 & k_{\rm{cat}_1} & 0 & 0 & 0 & 0 & 0 \\
      k_{\rm{on}_0} & 0 & 0 & -k_{\rm{off}_0}-k_{\rm{cat}_0} & 0 & 0 & 0 & 0 & 0 & 0 \\
      0 & k_{\rm{on}_1} & 0 & 0 & -k_{\rm{off}_1}-k_{\rm{cat}_1} & 0 & 0 & 0 & 0 & 0 \\
      0 & 0 & 0 & 0 & 0 & 0 & 0 & 0 & l_{\rm{cat}_0} & 0 \\
      0 & 0 & 0 & 0 & 0 & 0 & -l_{\rm{on}_0} & 0 & l_{\rm{off}_0} & l_{\rm{cat}_1} \\
      0 & 0 & 0 & 0 & 0 & 0 & 0 & -l_{\rm{on}_1} & 0 & l_{\rm{off}_1} \\
      0 & 0 & 0 & 0 & 0 & 0 & l_{\rm{on}_0} & 0 & -l_{\rm{cat}_0}-l_{\rm{off}_0} & 0 \\
      0 & 0 & 0 & 0 & 0 & 0 & 0 & l_{\rm{on}_1} & 0 & -l_{\rm{cat}_1}-l_{\rm{off}_1}
    \end{array}
  \right]\]
\end{small}

We will analyze the steady state locus of the resulting \CRS
by focusing on the structure of the kernel of the matrix
$\Sys = Y^t A_\kappa^t$ of the system.  Note that the network (\ref{eq:net_phospho_2_complete}) 
has only two terminal strong linkage classes, $\{S_2 + E\}$ and $\{S_0 + F\}$.
Also, ${\rm span} \{e_3, e_6\} \subseteq \ker(\Sys)$, where $e_i$ denotes the $i$-th canonical
vector of $\mathbb{R}^{10}$. A partition of the 10 complexes that satisfies 
Condition~\ref{cond:1} is given by  $I_1=\{1,4, 7, 9\}, \; I_2=\{2,5,8,10\}, \;
I_3=\{3\},$ and $I_4=\{6\}$. A corresponding basis of $\ker(\Sys)$, that is, one in which the 
$i$-th basis vector has support $I_i$, is:
\begin{small}
\[b^1=\left[\begin{array}{c}
    (k_{\rm{off}_0}+k_{\rm{cat}_0})k_{\rm{on}_1}k_{\rm{cat}_1}l_{\rm{on}_1}l_{\rm{on}_0}l_{\rm{cat}_0} \\
    0 \\
    0\\
    k_{\rm{on}_0}k_{\rm{on}_1}k_{\rm{cat}_1}l_{\rm{on}_1}l_{\rm{on}_0}l_{\rm{cat}_0} \\
    0 \\
    0\\
    k_{\rm{on}_0}k_{\rm{cat}_0}k_{\rm{on}_1}k_{\rm{cat}_1}l_{\rm{on}_1}(l_{\rm{cat}_0}+l_{\rm{off}_0}) \\
    0 \\
    k_{\rm{on}_0}k_{\rm{cat}_0}l_{\rm{on}_0}k_{\rm{on}_1}k_{\rm{cat}_1}l_{\rm{on}_1} \\
    0
    \end{array}\right]~,  ~
b^2=\left[\begin{array}{c}
    0\\
    k_{\rm{on}_0}k_{\rm{cat}_0}l_{\rm{on}_0}(k_{\rm{off}_1}+k_{\rm{cat}_1})l_{\rm{on}_1}l_{\rm{cat}_1}\\
    0\\
    0\\
    k_{\rm{on}_0}k_{\rm{cat}_0}l_{\rm{on}_0}k_{\rm{on}_1}l_{\rm{on}_1}l_{\rm{cat}_1}\\
    0\\
    0\\
    k_{\rm{on}_0}k_{\rm{cat}_0}l_{\rm{on}_0}k_{\rm{on}_1}k_{\rm{cat}_1}(l_{\rm{cat}_1}+l_{\rm{off}_1})\\
    0\\
    k_{\rm{on}_0}k_{\rm{cat}_0}l_{\rm{on}_0}k_{\rm{on}_1}k_{\rm{cat}_1}l_{\rm{on}_1}
    \end{array}\right]~, ~
b^3=e_3~,~b^4=e_6~.
\]
\end{small}
The structure of this basis $\{b^i\}$ implies that for $v\in \mathbb{R}^{10}$, 
 $v \in \ker(\Sys)$ if and only if $v$ satisfies the following binomial
equations:
\begin{equation}\label{eq:kernel_binomials}
\begin{array}{ll}
 b^1_{1}v_{4}-b^1_{4}v_{1}=0~, &  b^2_{2}v_{5}-b^2_{5}v_2=0~,\\
 b^1_{1}v_{7}-b^1_{7}v_{1}=0~, &  b^2_{2}v_{8}-b^2_{8}v_2=0~,\\
 b^1_{1}v_{9}-b^1_{9}v_{1}=0~, &  b^2_{2}v_{10}-b^2_{10}v_2=0~,\\
\end{array}
\end{equation}
Hence, any steady state 
of the 2-site phosphorylation system must satisfy 
the following equations in the species concentrations $x=(x_{S_0},x_{S_1},
\dots, x_{E},x_F)$:
\begin{equation} \label{eq:original_binomials}
\begin{array}{ll}
 b^1_{1}x_4-b^1_{4}x_8x_1=0~, &  b^2_{2}x_5-b^2_{5}x_8x_2=0~,\\
 b^1_{1}x_9x_2-b^1_{7}x_8x_1=0~, &  b^2_{2}x_9x_3-b^2_{8}x_8x_2=0~,\\
 b^1_{1}x_6-b^1_{9}x_8x_1=0~, &  b^2_{2}x_7-b^2_{10}x_8x_2=0~.\\
\end{array}
\end{equation}
  To check Condition~\ref{cond:3}, we consider the matrix $\Delta$ and the
  vector $\Theta_\kappa$:
  \begin{displaymath}
    \Delta = \left[ e_4-e_8-e_1~|~ e_9+e_2-e_8-e_1~|~ e_6-e_8-e_1~| ~
      e_5-e_8-e_2~|~ e_9+e_3-e_8-e_2~| ~ e_7-e_8-e_2 \right]
  \end{displaymath}
  \begin{displaymath}
    \Theta_\kappa = \left( \ln\frac{b_4^1}{b_1^1},\,
      \ln\frac{b_7^1}{b_1^1},\, \ln\frac{b_9^1}{b_1^1},\,
      \ln\frac{b_5^2}{b_2^2},\, \ln\frac{b_8^2}{b_2^2},\,
      \ln\frac{b_{10}^2}{b_2^2} 
    \right)~.
  \end{displaymath}
  It is straightforward to check that $\Delta$ has rank 6 and hence
  full rank. Thus Condition~\ref{cond:3} is trivially satisfied
  and does not pose any constraints on the rate constants.

Following the proof of Theorem~\ref{thm:toric_param}, we first will
parametrize the solution set of the following reduced system: 
\begin{equation} \label{eq:red_bins}
\begin{array}{ll}
 x_4-x_8x_1=0~, &  x_5-x_8x_2=0~,\\
 x_9x_2-x_8x_1=0~, &  x_9x_3-x_8x_2=0~,\\
 x_6-x_8x_1=0~, &  x_7-x_8x_2=0~.\\
\end{array}
\end{equation}
We are interested in an integer matrix $A$ such that
$\ker(A)={\rm im}\left(\Delta\right)$. 
One such matrix is 
$$A~=~
\left(\begin{array}{ccccccccc}
 0 & 1 & 2 & 1 & 2 & 1 & 2 & 1 & 0 \\
 0 & 0 & 0 & 1 & 1 & 1 & 1 & 1 & 1 \\
 1 & 1 & 1 & 1 & 1 & 1 & 1 & 0 & 0 \\
\end{array}\right)~.$$
This provides the following 3-dimensional parametrization of the
reduced system: 
$$(t_1,t_2,t_3) ~ \mapsto ~
	\left( t_3,~ t_1t_3,~ t_1^2t_3,~ t_1t_2t_3,~ t_1^2t_2t_3,~
	 t_1t_2t_3,~ t_1^2t_2t_3,~ t_1t_2,~ t_2  \right),$$
where $t_2$ is the concentration of the enzyme $F$, $t_1$ is the quotient of
the concentration of the enzyme $E$ divided by the concentration of the enzyme
$F$, and $t_3$ is the concentration of the substrate $S_0$.
Returning to the original binomials~\eqref{eq:original_binomials}, we have the following particular solution:
$$x^*_1=x^*_8=x^*_9=1, ~
x^*_2=\frac{b_7^1}{b_1^1}, ~
x^*_3=\frac{b_8^2b_7^1}{b_1^1b_2^2},~
x^*_4=\frac{b_4^1}{b_1^1}, ~
x^*_5=\frac{b_5^2b_7^1}{b_1^1b_2^2}, ~
x^*_6=\frac{b_9^1}{b_1^1},~
x^*_7=\frac{b_{10}^2b_7^1}{b_1^1b_2^2}~.
$$
Therefore we obtain the following 3-dimensional parametrization of the  
 positive steady state locus of \eqref{eq:original_binomials}, as predicted in
Theorem~\ref{thm:toric_param}:
\begin{align}
\label{eq:n=2_parametrization}
\mathbb{R}_{>0}^3 & \rightarrow \mathbb{R}_{>0}^9 \\
\nonumber
(t_1,t_2,t_3) & \mapsto 
	\left( t_3,~ \frac{b_7^1}{b_1^1}t_1t_3,~
	\frac{b_8^2b_7^1}{b_1^1b_2^2}t_1^2t_3,~ 
	\frac{b_4^1}{b_1^1}t_1t_2t_3,~
	\frac{b_5^2b_7^1}{b_1^1b_2^2}t_1^2t_2t_3,~
	\frac{b_9^1}{b_1^1}t_1t_2t_3,~
	 \frac{b_{10}^2b_7^1}{b_1^1b_2^2}t_1^2t_2t_3, ~
	t_1t_2,~ t_2 
	\right)~.
\end{align}
Recall that the values $b^i_j$ are polynomials in the rate
constants shown in the display of the vectors $b^1$ and $b^2$.
Finally, note that none of the calculations in this example depends
on the specific values of the rate constants; in particular, one partition 
works for all systems, so the hypothesis of 
Corollary~\ref{cor:toric_network_same_partition} holds.  
\end{example}

\subsection{More general sufficient conditions}

We show in Example~\ref{ex:ShiFein} below, extracted
from \cite{sf10}, that the conditions in Theorem~\ref{th:toric} are not 
necessary for a \CRS to have toric steady states; in other words, 
the converse of Theorem~\ref{th:toric} does not hold. However, the condition for
the steady state ideal to be generated by binomials always can be checked 
algorithmically via a Gr\"obner basis computation, as stated in the following lemma.

\begin{lemma}[Proposition 1.1.(a) of \cite{es96}] \label{lem:toric}
 Let $I$ be a binomial ideal, let $\prec$ be a monomial order, and 
let $G$ 
be the reduced Gr\"{o}bner basis of $I$ for that ordering.  Then $G$ consists of
binomials.
\end{lemma}

Lemma~\ref{lem:toric} is a basic result about binomial ideals which is due to Eisenbud and
Sturmfels~\cite{es96}; it is a result concerning polynomial linear combinations. Note however that
Theorem~\ref{th:toric} requires only \emph{linear algebra computations over $\mathbb R$}.
We make use of Lemma~\ref{lem:toric} in the following example. We will return to it 
later to show that Theorem~\ref{thm:necessary} below can be used to prove that
this system has toric steady states, without needing to compute a Gr\"{o}bner basis.

\begin{example}[Shinar and Feinberg network] \label{ex:ShiFein}
This example demonstrates that Condition~\ref{cond:1} is not necessary for
a \CRS to have toric steady states.
The network in Example~(S60) of the Supporting Online Material of the recent
article of Shinar and Feinberg is the following \cite{sf10}:
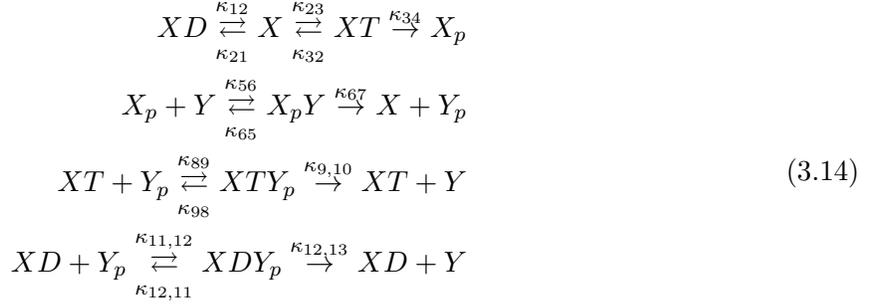
 \begin{equation}
\begin{split}
    XD
    \underset{\kappa_{21}}{\overset{\kappa_{12}}{\rightleftarrows}}
    X 
    \underset{\kappa_{32}}{\overset{\kappa_{23}}{\rightleftarrows}}
    XT
    \overset{\kappa_{34}}{\rightarrow}  
    X_p\\
    X_p + Y
    \underset{\kappa_{65}}{\overset{\kappa_{56}}{\rightleftarrows}}
    X_pY
    \overset{\kappa_{67}}{\rightarrow}
    X + Y_p\\
    XT + Y_p
    \underset{\kappa_{98}}{\overset{\kappa_{89}}{\rightleftarrows}}
    XTY_p
    \overset{\kappa_{9,10}}{\rightarrow}
    XT + Y\\
    XD + Y_p
    \underset{\kappa_{12,11}}{\overset{\kappa_{11,12}}{\rightleftarrows}}
    XDY_p
    \overset{\kappa_{12,13}}{\rightarrow}
    XD + Y\\    
  \label{eq:example_Feinberg}
\end{split}
\end{equation}
We denote by $x_1, x_2,\dots, x_9$ the concentrations of the species as follows:
$$x_{XD}=x_1, \; x_{X}=x_2, \; x_{XT}=x_3, \; x_{X_p}=x_4~,$$
$$x_{Y}=x_5, \; x_{X_pY}=x_6, \; x_{Y_p}=x_7, \; x_{XTY_p}=x_8, \;
x_{XDY_p}=x_9~.$$
Note that the numbering of the 13 complexes in the network is reflected in the 
names of the rate constants $\kappa_{ij}$.  The \CRS is the following:
\begin{equation} \label{eq:system_example_Feinberg}
 \begin{array}{ccl}
  \frac{dx_1}{dt} & = & -\kappa_{12}x_1+\kappa_{21}x_2-\kappa_{11,12}x_1x_7+(\kappa_{12,11}+\kappa_{12,13})x_9\\
  \frac{dx_2}{dt} & = & \kappa_{12}x_1+(-\kappa_{21}-\kappa_{23})x_2+\kappa_{32}x_3+\kappa_{67}x_6\\
  \frac{dx_3}{dt} & =  & \kappa_{23}x_2+(-\kappa_{32}-\kappa_{34})x_3-\kappa_{89}x_3x_7+(\kappa_{98}+\kappa_{9,10})x_8\\
  \frac{dx_4}{dt} & = & \kappa_{34}x_3-\kappa_{56}x_4x_5+\kappa_{65}x_6\\
  \frac{dx_5}{dt} & = & -\kappa_{56}x_4x_5+\kappa_{65}x_6+\kappa_{9,10}x_8+\kappa_{12,13}x_9\\
  \frac{dx_6}{dt} & = & \kappa_{56}x_4x_5+(-\kappa_{65}-\kappa_{67})x_6\\
  \frac{dx_7}{dt} & = & \kappa_{67}x_6-\kappa_{89}x_3x_7+\kappa_{98}x_8-\kappa_{11,12}x_1x_7+\kappa_{12,11}x_9\\
  \frac{dx_8}{dt} & = & \kappa_{89}x_3x_7+(-\kappa_{98}-\kappa_{9,10})x_8\\
  \frac{dx_9}{dt} & = & \kappa_{11,12}x_1x_7+(-\kappa_{12,11}-\kappa_{12,13})x_9
 \end{array}
\end{equation}
The reduced Gr\"{o}bner basis with respect to the lexicographical order $x_1 >
x_2 > x_4 > x_5 > x_6 > x_8 > x_9 > x_3 > x_7$ 
consists of the following binomials:
\begin{small}
\begin{equation} \label{eq:Groebner_Feinberg}
 \begin{array}{ccl}
  g_1 & = &
[\kappa_{89}\kappa_{12}\kappa_{23}\kappa_{9,10}(\kappa_{12,11}+\kappa_{12,13})+\kappa_{11,12}\kappa_{21}\kappa_{12,13}(\kappa_{98}+\kappa_{
9,10})(\kappa_{32}+\kappa_{34})]x_3x_7+\\
      &   & \quad +[-\kappa_{23}\kappa_{34}\kappa_{12}(\kappa_{12,11}+\kappa_{12,13})(\kappa_{98}+\kappa_{9,10})]x_3\\
  g_2 & = & [-\kappa_{11,12}\kappa_{21}\kappa_{34}(\kappa_{98}+\kappa_{9,10})(\kappa_{32}+\kappa_{34})]x_3+\\
      &   & \quad
+[\kappa_{11,12}\kappa_{21}\kappa_{12,13}(\kappa_{98}+\kappa_{9,10})(\kappa_{32}+\kappa_{34})+\kappa_{12}\kappa_{23}\kappa_{89}\kappa_{9,10
}(\kappa_{12,11}+\kappa_{12,13})]x_9\\
  g_3 & = & [-\kappa_{23}\kappa_{34}\kappa_{89}\kappa_{12}(\kappa_{12,11}+\kappa_{12,13})]x_3+\\
      &   & \quad
+[\kappa_{23}\kappa_{9,10}\kappa_{89}\kappa_{12}(\kappa_{12,11}+\kappa_{12,13})+\kappa_{11,12}\kappa_{21}\kappa_{12,13}(\kappa_{98}+\kappa_{
9,10})(\kappa_{32}+\kappa_{34})]x_8\\
  g_4 & = & \kappa_{67}x_6-\kappa_{34}x_3\\
  g_5 & = & \kappa_{56}\kappa_{67}x_4x_5+\kappa_{34}(-\kappa_{65}-\kappa_{67})x_3\\
  g_6 & = & \kappa_{23}x_2+(-\kappa_{32}-\kappa_{34})x_3\\
  g_7 & = & -\kappa_{21}(\kappa_{32}+\kappa_{34})x_3+\kappa_{12}\kappa_{23}x_1\\
 \end{array}
\end{equation}
\end{small}
Therefore, the network has toric steady states (for any choice of positive reaction rate constants)
because the steady state ideal
 can be generated by 
$g_1,~g_2,~ \dots,~g_7$.  However, we claim that this \CRS does not satisfy 
Condition~\ref{cond:1}. In fact,
for any rate constants, 
it is not possible to find a partition $I_1,I_2, \dots, I_6 \subseteq 
\{1,2, \dots, 13\}$
such that $\ker(\Sys)$ has a basis $\{b^1, b^2, \dots, b^6\}$ with $\supp(b^i)=I_i$.  
This can be seen by noting that 
the kernel of $\Sys$ can be generated as follows:
\begin{small}
\begin{align} 
\ker(\Sys)=
&\left\langle  e_4,~ e_7,~ e_{10},~ e_{13},~\left(\frac{\kappa_{21}\kappa_{12,13}(\kappa_{32}+\kappa_{34})}{\kappa_{23}\kappa_{34}\kappa_{12}}\right) e_1+
\left(\frac{\kappa_{12,13}(\kappa_{32}+\kappa_{34})}{\kappa_{23}\kappa_{34}}\right) e_2 + 
\left(\frac{\kappa_{12,13}}{\kappa_{34}}\right) e_3 + \right.\\ \nonumber
& \quad \quad \left. +
\left(\frac{(\kappa_{65}+\kappa_{67})\kappa_{12,13}}{\kappa_{67}\kappa_{56}}\right) e_5 +
\left(\frac{\kappa_{12,13}}{\kappa_{67}}\right) e_6 +
\left(\frac{(\kappa_{12,11}+\kappa_{12,13})}{\kappa_{11,12}}\right) e_{11} +
e_{12},\right.\\ \nonumber
& \quad \left.\left(
 \frac{\kappa_{21}\kappa_{9,10}(\kappa_{32}+\kappa_{34})}{\kappa_{23}\kappa_{34}\kappa_{12}}\right) e_1 +
\left(\frac{\kappa_{9,10}(\kappa_{32}+\kappa_{34})}{\kappa_{23}\kappa_{34}}\right) e_2 +
\left(\frac{\kappa_{9,10}}{\kappa_{34}}\right) e_3 + \right.\\ \nonumber
& \quad \quad \left. +
\left(\frac{(\kappa_{65}+\kappa_{67})\kappa_{9,10}}{\kappa_{67}\kappa_{56}}\right) e_5  +
\left(\frac{\kappa_{9,10}}{\kappa_{67}}\right) e_6 + 
\left(\frac{\kappa_{98}+\kappa_{9,10}}{\kappa_{89}}\right) e_8 +
e_9 \right\rangle ~.
\end{align}
\end{small}

\end{example}
Our next result, 
Theorem~\ref{thm:necessary}, will generalize
Theorem~\ref{th:toric} by 
giving a stronger condition that guarantees that the steady state locus is generated
by binomials.  We first need to generalize Conditions~\ref{cond:1},~\ref{cond:2},
and~\ref{cond:3} to any (finite) polynomial system.

First we must introduce some notation. 
For polynomials $F_1, F_2, \dots, F_{s'} \in \R[x_1,x_2, \dots, x_s]$, we 
denote by $x^{y_1}, x^{y_2}, \dots, x^{y_{m'}}$ the monomials that occur 
in these polynomials; that is, there exist $F_{ij} \in \mathbb{R}$ such that
$F_i(x) \, = \sum_{j =1}^{m'} F_{ij} x^{y_j}$  for $ i=1,2, \dots, s'$.
We can write the polynomial system $F_1(x) =F_2(x) = \dots = F_{s'}(x) =0$ as
\begin{equation} \label{eq:sysprime}
 {\Sys}' \cdot {\Psi}'(x) ~=~ 0~, 
\end{equation}
where $\Sys'= (F_{ij})\in \R^{s'\times m'}$ is the coefficient matrix and
$\Psi'(x) = (x^{y_1},x^{y_2}, \dots, x^{y_{m'}})^t$. We will let $d'$ denote the
dimension of  $ \ker(\Sys')$. 

\begin{condition} \label{cond:11}
We say that the polynomial system~\eqref{eq:sysprime} satisfies
Condition~\ref{cond:11} if there exists
    a partition $I_1,I_2, \dots, I_{d'}$ of $\{1,2, \dots, m'\}$ and a
    basis $b^1, b^2, \ldots, b^{d'} \in \R^{m'}$ of $\ker(\Sys')$ such that ${\rm
      supp}(b^i) = I_i$.
  
\end{condition}

 \begin{condition} \label{cond:22}
    Consider a polynomial system~\eqref{eq:sysprime} that satisfies
    Condition~\ref{cond:11} for the partition $I_1,I_2,
    \dots, I_{d'}$ of $\{1,2, \dots, m'\}$ and a basis $b^1, b^2,
    \ldots, b^{d'} \in \R^{m'}$ of $\ker(\Sys')$ (with ${\rm supp}(b^i) =
    I_i$).
      We say that the system satisfies \textbf{additionally}
    Condition~\ref{cond:22}, if for all $j \in \{1,2, \dots, d'\}$, the
    nonzero entries of $b^j$ have the same sign.
  \end{condition}

  As before, we collect the differences of exponent vectors as columns of a
  matrix 
  \begin{equation}
    \label{eq:def_Delta'}
    \Delta' := \left[
      \left(y_{j_1}-y_{j_2}\right)^t
    \right]_{\forall j_1,\, j_2 \in I_j,\, \forall 1 \leq j \leq d'}
  \end{equation}
  and define the (row) vector
  \begin{equation}
    \label{eq:def_Theta'}
    \Theta':= \left(
      \ln \frac{b^j_{j_2}}{b^j_{j_1}} 
    \right)_{\forall j_1,\, j_2 \in I_j,\, \forall 1 \leq j \leq d'}\ .
  \end{equation}

  \begin{condition} \label{cond:33}
    Consider a polynomial system~\eqref{eq:sysprime} which satisfies Conditions~\ref{cond:11} and~\ref{cond:22}.
 Let $U'$ be a matrix with integer entries whose columns form a
    basis of the kernel of $\Delta'$. We say that this system satisfies \textbf{additionally}
    Condition~\ref{cond:33}, if the following holds:
    \begin{equation*}
      \Theta'\, U' = 0\ .
    \end{equation*}
  \end{condition}

We then have the following sufficient conditions:
\begin{theorem}\label{thm:necessary}
Consider a \CRS with $m$ complexes and assume that there exist monomials 
$\mathbf{x}^{\alpha_{1}}, \mathbf{x}^{\alpha_{2}}, \dots, \mathbf{x}^{\alpha_{\ell}}$ and indices 
$i_1, i_2, \dots, i_{\ell}$, with $\{i_1,i_2,\dots, i_{\ell}\} \subseteq \{1,2, \dots, s\}$,
such that Condition~\ref{cond:11} holds for the enlarged polynomial system 
\[f_{1} = \dots = f_{s} = \mathbf{x}^{\alpha_{1}}f_{i_1}=  \dots =\mathbf{x}^{\alpha_{\ell}}f_{i_{\ell}}=0.\]
Then the steady state ideal $J_{\Sys \psi}$ is binomial. 

Moreover, the system has positive (toric) steady states if and only if 
Conditions~\ref{cond:22} and~\ref{cond:33} hold additionally
for the enlarged system.
\end{theorem}

This theorem can be proved following the lines of the proof of Theorem~\ref{th:toric} for the 
enlarged system defined in the statement. It is important to note that the ideal 
$\langle f_1, f_2, \dots,f_s \rangle$ equals the ideal $\langle f_1, \dots,f_s, 
\mathbf{x}^{\alpha_{1}}f_{i_1},  \dots,\mathbf{x}^{\alpha_{\ell}}f_{i_{\ell}} \rangle$.

With similar proof as in Theorem~\ref{thm:toric_param}, we moreover have:
\begin{theorem}\label{thm:toricgeneral_param}
 Under the hypotheses of Theorem~\ref{thm:necessary},  the steady state locus can be
parametrized by monomials in the concentrations.
\end{theorem}

\begin{remark}\label{rmk:conditions'}
  As with Conditions~\ref{cond:1},~\ref{cond:2},
  and~\ref{cond:3}, checking the Conditions in the statement of
  Theorem~\ref{thm:necessary} involves linear algebra computations
  over $\R$ for fixed rate constants or over $\Q(k_{ij})$ for generic
  parameters, once the monomials $x^{\alpha_i}$ are given. In small
  cases,  such monomials can be guessed. In the following example,
  they were  traced in the standard algorithm for the computation of a
  Gr\"obner basis of the ideal $\langle f_1, \dots, f_s \rangle$.
\end{remark}

We end this section by returning to Example~\ref{ex:ShiFein}.

\begin{example}[Shinar and Feinberg network, continued] \label{ex:return_ShiFein}
Consider the system of equations:
\begin{equation}\label{eq:enlarged}
\left\lbrace \begin{array}{l}
 f_1=0\\
 f_2=0\\
 \vdots\\
 f_9=0\\
 x_7f_1=0\\
 x_7f_3=0\\
 x_7f_8=0\\
 x_7f_9=0
\end{array}\right.\;,
\end{equation}
This enlarged system satisfies Conditions~\ref{cond:11} and~\ref{cond:22} for the following partition:
\begin{small}
$$I_1=\{4\}, \, I_2=\{10\}, \, I_3=\{13\}, \, I_4=\{14,15\}, \, I_5=\{16,17\},
I_6=\{1,2,3,5,6,7,8,9,11,12\}~$$
\end{small}
and the following basis $b^1, b^2, \dots, b^6$ of its kernel verifying $\supp(b^j)=I_j$:
\begin{small}
\begin{align*}
b^1= & e_4~, \quad
 \, b^2=e_{10}~,  \quad
 \, b^3=e_{13}~, \quad 
 \, b^4=(k_{12,11}+k_{12,13})e_{14}+k_{11,12}e_{15}, \quad b^5=(k_{98}+k_{910})e_{16}+k_{89}e_{17},\\
b^6=&(k_{12}k_{23}k_{89}k_{9,10}(k_{12,11}+k_{12,13})+k_{21}k_{11,12}k_{12,13}(k_{32}+k_{34})(k_{98}+k_{9,10}))k_{21}(k_{32}+k_{34})k_{56}k_{67}e_1+\\
	&\quad (k_{12}k_{23}k_{89}k_{9,10}(k_{12,11}+k_{12,13})+k_{21}k_{11,12}k_{12,13}(k_{32}+k_{34})(k_{98}+k_{9,10}))k_{12}(k_{32}+k_{34})k_{56}k_{67}e_2+\\
	&\quad (k_{12}k_{23}k_{89}k_{9,10}(k_{12,11}+k_{12,13})+k_{21}k_{11,12}k_{12,13}(k_{32}+k_{34})(k_{98}+k_{9,10}))k_{12}k_{23}k_{56}k_{67}e_3+\\
	& \quad (k_{12}k_{23}k_{89}k_{9,10}(k_{12,11}+k_{12,13})+k_{21}k_{11,12}k_{12,13}(k_{32}+k_{34})(k_{98}+k_{9,10}))k_{12}k_{23}k_{34}(k_{65}+k_{67})e_5+\\
	&\quad (k_{12}k_{23}k_{89}k_{9,10}(k_{12,11}+k_{12,13})+k_{21}k_{11,12}k_{12,13}(k_{32}+k_{34})(k_{98}+k_{9,10}))k_{12}k_{23}k_{34}k_{56}e_6+\\
	&\quad k_{12}^2k_{23}k_{34}(k_{32}+k_{34})k_{56}k_{67}(k_{98}+k_{9,10})(k_{12,11}+k_{12,13})e_7+\\
	&\quad k_{12}^2k_{23}^2k_{34}k_{56}k_{67}(k_{98}+k_{9,10})(k_{12,11}+k_{12,13})e_8+
k_{12}^2k_{23}^2k_{34}k_{56}k_{67}k_{89}(k_{12,11}+k_{12,13})e_9+\\
	&\quad k_{12}k_{21}k_{23}k_{34}(k_{32}+k_{34})k_{56}k_{67}(k_{98}+k_{9,10})(k_{12,11}+k_{12,13})e_{11}+\\
	&\quad k_{12}k_{21}k_{23}k_{34}(k_{32}+k_{34})k_{56}k_{67}(k_{98}+k_{9,10})k_{11,12}e_{12}~.
\end{align*}
\end{small}
In addition to the monomials already occurring in $f_1,f_2, \dots, f_9$, the following
$4$ monomials are also in the augmented system: $x^{y_{14}} = x_1 x_7^2,$  
$x^{y_{15}} = x_9 x_7,$  
$x^{y_{16}} = x_3 x_7^2,$ and 
$x^{y_{17}} = x_8 x_7$. 
By Theorem~\ref{thm:necessary}, the system has toric steady states.   
Recall that the  binomials $g_1,g_2, \dots, g_7$  in equation~\eqref{eq:Groebner_Feinberg} generate the ideal
$\langle f_1,~f_2,~ \dots, f_9 \rangle =$ $\langle f_1,~ f_2,~ \dots, f_9,~
 x_7 f_1, ~x_7 f_3, ~x_7 f_8, ~x_7 f_9 \rangle$.
 We can see immediately that there are positive steady states for any choice of positive
rate constants, and so there is no need to check Condition~\ref{cond:33}.

\end{example}
\section{The $n$-site phosphorylation system has toric steady states}\label{sec:multisite_toric}
In this section we introduce the $n$-site phosphorylation system (under the assumption of a distributive and sequential mechanism).  
To show that these systems have toric steady states, we apply Theorem~\ref{th:toric};
this generalizes Example~\ref{ex:trunning} (the $n=2$
case).  Further, we note that 
the parametrization of the steady state locus given by 
Theorem~\ref{thm:toric_param} is implicit in work of 
Wang and Sontag \cite{WangSontag}.

\subsection{The $n$-site phosphorylation system}
We now define the $n$-site phosphorylation system (also called a ``multiple futile cycle'') 
$\Sigma_{n}(\kappa,\mathcal{C})$, which depends on a choice of rate constants
$\kappa\in\Rplus^{6n}$ and values of the conservation relations $\mathcal{C} =
\left( \Etot, \Ftot, \Stot \right)\in \Rplus^{3}$.  As in the earlier example 
of the 1-site network~\eqref{eq:net_phospho_1_complete} and the 2-site 
network~\eqref{eq:net_phospho_2_complete}, we will make the
assumption of a ``distributive'' and ``sequential'' mechanism
  (see, for example, \cite{ConradiUsing}). As discussed in the
  Introduction, this  $n$-site phosphorylation system is of great
  biochemical importance: it is a recurring network motif in many
  networks describing processes as diverse as intracellular signaling
  (e.g.\ MAPK signaling with $n=2$ and $n=3$), cell cycle control
  (e.g.\ Sic1 with $n=9$), and cellular differentiation (e.g.\ NFAT with
  $n=13$).

Following notation of Wang and Sontag \cite{WangSontag}, the 
$n$-site phosphorylation system arises from the following reaction network:
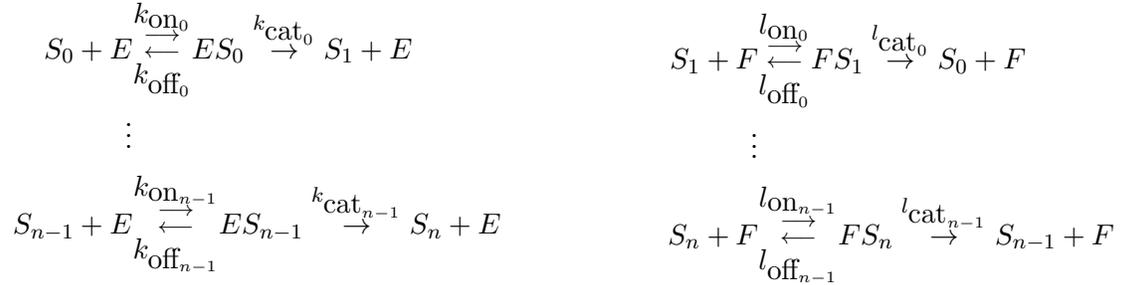
\begin{multicols}{2}
\begin{align*}
S_0+E &\arrowschem{k_{\mbox{on}_0}}{k_{\mbox{off}_0}} ES_0
\stackrel{k_{\mbox{cat}_0}}{\rightarrow} S_1+ E \\
\vdots &   \\
S_{n-1}+E &\arrowschem{k_{\mbox{on}_{n-1}}}{k_{\mbox{off}_{n-1}}}ES_{n-1}
\stackrel{k_{\mbox{cat}_{n-1}}}{\rightarrow} S_n+ E 
\end{align*}

\begin{align*}
S_1+F &\arrowschem{l_{\mbox{on}_0}}{l_{\mbox{off}_0}} FS_1
\stackrel{l_{\mbox{cat}_0}}{\rightarrow} S_0+ F \\
\vdots &   \\
S_n+F &\arrowschem{l_{\mbox{on}_{n-1}}}{l_{\mbox{off}_{n-1}}} FS_n
\stackrel{l_{\mbox{cat}_{n-1}}}{\rightarrow} S_{n-1}+ F 
\end{align*}
\end{multicols}
\noindent We see that the $n$-site network has $3n+3$ chemical species 
$S_0,\dots, S_n$, $ES_0,\dots, ES_{n-1}$, $FS_1,\dots, FS_n$, $E,$ and $F$, 
so we write a concentration vector as 
$x=(\s_0, \dots, \s_n, \es_0,\dots, \es_{n-1}, \fs_1, \dots, \fs_n, e,f)$, 
which is a positive vector of length $3n+3$.   These species
comprise $4n+2$ complexes, and there are $6n$ reactions.  Each reaction
has a reaction rate, and we collect these in the vector of rate constants
$\kappa = \left(k_{\rm{on}_0}, \dots,  l_{\rm{cat}_{n-1}} \right)\in \mathbb{R}^{6n}_{>0}$.  

For our purposes, we will introduce the following numbering for the complexes 
(which is compatible with the numbering in Examples~\ref{ex:n=1} and~\ref{ex:trunning}):
\begin{multicols}{2}
\begin{align*}
1 &\rightleftarrows n+2
\stackrel{ }{\rightarrow} 2 \\
2 &\rightleftarrows n+3
\stackrel{}{\rightarrow} 3 \\
\vdots &   \\
n &\rightleftarrows 2n+1
\stackrel{}{\rightarrow} n+1
 \end{align*}

 \begin{align*}
2n+3 &\rightleftarrows 3n+3
\stackrel{}{\rightarrow} 2n+2 \\
2n+4 &\rightleftarrows 3n+4
\stackrel{}{\rightarrow} 2n+3 \\
\vdots &   \\
3n+2 &\rightleftarrows 4n+2
\stackrel{}{\rightarrow} 3n+1
\end{align*}
\end{multicols}
The conservation relations here correspond to the fact that the total 
amounts of free and bound enzyme or substrate remain constant.  That is,
the following three conservation values $\mathcal{C}=\left(\Etot,\Ftot,\Stot\right) \in \R_{>0}^3$ 
remain unchanged as the dynamical system
progresses: 
\begin{align}
\label{eqn:conservation}
E_{\mbox{tot}}&=\e+\sum_{i=0}^{n-1}\es_i~,\notag \\
F_{\mbox{tot}}&=\f+\sum_{i=1}^{n}\fs_i~, \\
S_{\mbox{tot}}&=\sum_{i=0}^n \s_i+\sum_{i=0}^{n-1}\es_i+\sum_{i=1}^{n}\fs_i~.\notag
\end{align}
Any choice of these three values defines a  bounded stoichiometric compatibility 
class of dimension $3n$:
\begin{align*}
 \mathcal{P}_{\mathcal{C}} ~=~ 
	\left\{ x \in \Rnn^{3n+3} ~|~ 
	 \text{the conservation equations~\eqref{eqn:conservation} hold} \right\}.
\end{align*}
Note that the right hand side of each of the
three conservation relations \eqref{eqn:conservation} is defined by
a vector $z_i\in S^{\perp}$, $i=1$, $2$, $3$. These vectors play an
important role in the following remark and in Lemma~\ref{lem:nobss}
below.

\begin{remark}[Positive steady states by fixed point arguments]
  As indicated in Remark~\ref{rem:fixed-point-arguments} one may
  deduce the existence of at least one positive steady state in each
  stoichiometric compatibility class $\mathcal{P}_{\mathcal{C}}$
  (defined by positive $\mathcal{C}$) by fixed-point arguments,
  provided (i) $\mathcal{P}_{\mathcal{C}}$ is bounded and (ii) there
  are no boundary steady states in any stoichiometric compatibility
  class $\mathcal{P}_{\mathcal{C}}$ (with ${\mathcal{C}} \in
  \R_{>0}^3$). Point (i) follows from the definition of
  ${\mathcal{P}_\mathcal{C}} \in \R_{>0}^3$ given above. With respect
  to (ii), we point to Lemma~\ref{lem:nobss} below (which can be
  established by a straightforward generalization of the analysis due
  to Angeli, De~Leenheer, and Sontag in Examples 1 and 2 in
  \cite[\S~6]{fein-042}).
  \begin{lemma}\label{lem:nobss}
    Let $x^* \in \R_{\ge 0}^s - \R_{>0}^s$ be a boundary steady
    state. Set $\Lambda:=\{ i \in \{1, \dots, s\} : x^*_i=0\}$. Then,  
    $\Lambda$ contains the support of at least one of the vectors
    $z_i\in S^{\perp}$ defining the conservation relations~\eqref{eqn:conservation}.
    Therefore, there are no boundary steady states in any stoichiometric
    compatibility class $\mathcal{P}_{\mathcal{C}}$ with ${\mathcal{C}}
    \in \R_{>0}^3$.
  \end{lemma}
\end{remark}



We will see in Theorem~\ref{thm:n-site} that the steady state locus in this system is 3-dimensional. 
A forthcoming work 
will concern the question of how many times the steady state locus intersects 
 the relative interior of a compatibility class $ \mathcal{P}_{\mathcal{C}}$ for multisite phosphorylation 
systems \cite{article2}. 

\subsection{Results}
For the $n$-site phosphorylation system, we will call its \CSMatrix ~$\Sys_n$, and we will let
$G_n$ denote the 
underlying digraph of the chemical reaction network.
In order to apply the results of Section~\ref{sec:toric} to this system, 
we now aim to exhibit a specific basis of the kernel of $\Sys_n$ that satisfies
Condition~\ref{cond:1}.  
We begin by describing 
the rows of $\Sys_n~:=~Y^t \cdot A_{\kappa}^t~$ as linear 
combinations of the rows of $A_{\kappa}^t$. Recall that 
$A_{\kappa}$ is minus the Laplacian matrix of the associated digraph.
Letting $R(i)$ represent the $i$-th row of $A_{\kappa}^t$, we have:
{\small
\begin{equation} 
  \label{eq:Sys_n}
  \Sys_n~:=~Y^t \cdot A_{\kappa}^t~= 
  \left[
    \begin{array}{c}
     R(1)+R(2n+2)\\
     R(2)+R(2n+3)\\
     \vdots\\
     R(n+1)+R(3n+2)\\
     \hline
     R(n+2)\\
     \vdots\\
     R(2n+1)\\
     \hline
     R(3n+3)\\
     \vdots\\
     R(4n+2)\\
     \hline
     R(1)+R(2)+ \dots + R(n+1)\\
     R(2n+2)+ \dots +R(3n+2)\\
    \end{array}  
  \right] \in \mathbb{R}^{(3n+3)\times (4n+2)}
\end{equation}}
Our next aim is to exhibit a submatrix of $\Sys_n$ that shares the same kernel as $\Sys_n$.
The only relations that exist among the rows of $A_{\kappa}^t$ arise from 
the fact that the sum of the rows in each of the four blocks equals zero.  Consequently, it is  
straightforward to check that
\begin{equation*}
 {\rm rank}(\Sys_n)~=~3n~.
\end{equation*}
Moreover, if we delete any of the first $3n+1$ rows and the last two rows of $\Sys_n$, we obtain 
a new matrix that has maximal rank. As we are interested in describing the kernel of 
$\Sys_n$, we will discard the first and the last two rows, and we will focus on the 
resulting submatrix. Furthermore, as the $(n+1)$-st and $(2n+2)$-nd columns on 
$\Sys_n$ are equal to zero, we already know that $e_{n+1}$ and $e_{2n+2}$, the 
$(n+1)$-st and $(2n+2)$-nd canonical basis vectors of $\mathbb{R}^{4n+2}$, 
belong to $\ker(\Sys_n)$. Hence we can now focus on an even smaller submatrix of 
$\Sys_n$ obtained by deleting the first and the last two rows, and the 
$(n+1)$-st and $(2n+2)$-nd columns. We will call this submatrix $\Sys'_n$, and 
we will denote by $C(j)$ the column of $\Sys'_n$ which corresponds to the $j$-th column 
of $\Sys_n$ after deleting the first row and the last two (for example, $C(n+2)$ will 
represent the $(n+1)$-st column of $\Sys'_n$). Then, if we call $\Sys''_n$ the 
submatrix of $\Sys'_n$ formed by its first $3n$ columns, the system $\Sys'_n v~=~0$ is 
equivalent to the following one:
\begin{small}
\begin{equation}\label{eq:SysCramer}
\Sys''_n\left[ \begin{array}{c}
                                  v_1\\
				  \vdots\\
				  v_{3n}\\\end{array}\right]=-\left[
\begin{array}{ccc}
        C(3n+3) & \dots & C(4n+2)\\
       \end{array}\right] \left[ \begin{array}{c}
				  v_{3n+1}\\
				  \vdots\\
				  v_{4n}\\
                                 \end{array}\right]~.
\end{equation}
\end{small}

Let us call
\begin{equation}\label{eq:D}
 D~:=~\det(\Sys''_n)~.
\end{equation}
If $D \neq 0$, then we can use Cramer's rule to solve system~\eqref{eq:SysCramer}. 
In fact, we will show in Proposition~\ref{prop:Sign+Dnonzero} that this is 
the case and that we can find solutions to the system $\Sys_n w~=~0$ such that 
all the nonzero entries have the same sign.

Next we introduce a partition and a set of basis vectors $b^i$ that will be used to show that 
the $n$-site system satisfies Condition~\ref{cond:1}.  The partition $I_1, I_2, \dots, I_{n+2}$ of 
$\{1,2, \dots, 4n+2\}$ is the following:

\begin{small}
\begin{equation}\label{eq:partition}
I_j~=~\{j, \, n+j+1, \, 2n+j+2, \, 3n+j+2\} \,~ ({\rm for} \, 1 \leq j \leq n), \quad I_{n+1}~=~\{n+1\}, \quad I_{n+2}~=~\{2n+2\}~.
\end{equation}
\end{small}

The entries in our vectors $b^i$ will be certain determinants.  More precisely, let $D_{\ell(j)}$ be minus the determinant of the matrix obtained by replacing $C(\ell(j))$ 
by $C(3n+j+2)$ in $\Sys''_n$, for 
$\ell(j)=j, ~
n+j+1, ~
2n+j+2,$
 where $1 \leq j \leq n$:
\begin{small}
\begin{equation}\label{eq:D_j}
 D_{\ell(j)}=-\det\left(\left[C(1) | \dots | \overset{\overset{\ell(j)}{\downarrow}}{C(3n+j+2)} | \dots | C(3n+2) \right]\right)~.
\end{equation}
\end{small}

Note that $D$, $D_j, D_{n+j+1},$ and $D_{2n+j+2}$, for $1\leq j \leq n$, define polynomial 
functions of $\kappa$ on $\Rplus^{6n}$. We will show in Proposition~\ref{prop:Sign+Dnonzero} 
that these functions $D$, $D_j$, $D_{n+j+1}$, and $D_{2n+j+2}$ are nonzero and have the same sign, 
for $1 \leq j \leq n$. 

Now we may define the vectors $b^1,b^2, \dots, b^n$ of $\Rplus^{4n+2}$ by:
\begin{equation}\label{eq:basisKer}
 (b^j)_i=\left\lbrace\begin{array}{ll}
  D_{j} & {\rm if} \; i=j\\
  D_{n+j+1} & {\rm if} \; i=n+j+1\\
  D_{2n+j+2} & {\rm if} \; i=2n+j+2\\
  D & {\rm if} \, i=3n+j+2\\
  0 & {\rm otherwise}\\
 \end{array}\right.~,
\end{equation}
for $1\leq i\leq 4n+2$, where $1\leq j \leq n$.

We are now equipped to state our main result in this section.
\begin{theorem}\label{thm:n-site}
 The $n$-site phosphorylation system has toric steady states. The 
steady state locus has dimension 3 and can be parametrized by

\begin{small}
\begin{align*}
\mathbb{R}^3~ \rightarrow~ & \mathbb{R}^{3n+3} \\
(t_1,t_2,t_3) ~ \mapsto ~ &\left(t_3,~ \frac{D_{2n+3}}{D_{1}}t_1t_3,~\dots,~ 
\frac{D_{2n+3}}{D_{1}}\dots\frac{D_{3n+2}}{D_{n}}t_1^n t_3,~ \frac{D_{n+2}}{D_{1}}t_1t_2t_3,~ 
\dots,~ \frac{D_{n+2}}{D_{1}}\dots\frac{D_{2n+1}}{D_{n}}t_1^nt_2t_3,~ \right.\\
&\quad \left. 
\frac{D}{D_{1}}t_1t_2t_3,~ 
\dots,~ \frac{D}{D_{n}}\frac{D_{2n+3}}{D_{1}}
\dots\frac{D_{3n+1}}{D_{n-1}}t_1^nt_2t_3,~ t_1t_2,~t_2\right)~.
\end{align*}
\end{small}

Moreover, the system satisfies Condition~\ref{cond:1} with the partition $I_1, I_2, \dots, I_{n+2}$ 
described in~\eqref{eq:partition} and the basis $\{b^1, \dots , b^n\}\cup \{e_{n+1}, e_{2n+2}\}$ 
where the vectors $b^j$ are defined in~\eqref{eq:basisKer} and $e_{n+1}$ and $e_{2n+2}$ are the 
$(n+1)$-st and $(2n+2)$-nd vectors of the canonical basis of $\mathbb{R}^{4n+2}$.
In addition, it satisfies Conditions~\ref{cond:2} and~\ref{cond:3}.

In particular, 
\begin{small}
 $$\tilde{x}=\left(1,~ \frac{D_{2n+3}}{D_{1}},~\dots,~ 
\frac{D_{2n+3}}{D_{1}}\dots\frac{D_{3n+2}}{D_{n}},~ \frac{D_{n+2}}{D_{1}},~ 
\dots,~ \frac{D_{n+2}}{D_{1}}\dots\frac{D_{2n+1}}{D_{n}},~ \frac{D}{D_{1}},~ 
\dots,~\frac{D}{D_{n}}\frac{D_{2n+3}}{D_{1}}
\dots\frac{D_{3n+1}}{D_{n-1}},~1,~1\right)$$
\end{small}
 is an explicit positive steady state of the system.
\end{theorem}

We remark that the parametrization given in the statement of this theorem,
which is one of the possible
parametrizations provided by Theorem~\ref{thm:toric_param},
 gives systematically
what Wang and Sontag obtained ``by hand'' in \cite{WangSontag}. 
We note that the fact that this variety (the steady state locus) has a rational parametrization is a
special case of a rational parametrization theorem for general
 multisite post-translational modification systems as analyzed by Thomson and Gunawardena \cite{TG}.

\subsection{Proof of Theorem~\ref{thm:n-site}}

We start with the following proposition:

\begin{proposition}\label{prop:Sign+Dnonzero}
Let $D$ be the determinant defined in~\eqref{eq:D}, and let $D_{j}$, $D_{n+j+1}$, and 
$D_{2n+j+2}$ be as in~\eqref{eq:D_j}, for $1 \leq j \leq n$. 
Then each polynomial function $D, D_{j }, D_{n+j+1 }, D_{2n+j+2}:\Rplus^{6n}\rightarrow \R$ 
for $1 \leq j \leq n$, never vanishes, and these functions all have the same constant sign on $\Rplus^{6n}$.
\end{proposition}

\begin{proof}
For this proof, we will denote by $R(i)$ the $i$-th row of the matrix obtained from 
$A_{\kappa}^t$ after deleting columns $n+1$ and $2n+2$.  (Note that this notation
differs slightly from that introduced in equation~\eqref{eq:Sys_n}.)  The proof has two steps: 
first we demonstrate that $D\neq 0$ on the positive orthant, and then we show that the other functions $D_{j}$, $D_{n+j+1}$, and 
$D_{2n+j+2}$ are also nonzero on the positive orthant and that their signs coincide with that of $D$.

To prove that $D \neq 0$ on $\Rplus^{6n}$, we proceed by induction on $n$.  
First, if $n=1$, we have:
\begin{small}
\begin{equation*}
\Sys''_1=\left[\begin{array}{ccc}
                      0 & k_{\rm{cat}_0} & -l_{\rm{on}_0}\\
                      k_{\rm{on}_0} & -k_{\rm{off}_0}-k_{\rm{cat}_0} & 0 \\
                      0 & 0 & l_{\rm{on}_0}\\
                     \end{array}\right]~.
\end{equation*}
\end{small} 
In this case, $D=-k_{\rm{on}_0}k_{\rm{cat}_0}l_{\rm{on}_0}\neq 0$, as we wanted.

For the $n>1$ case, we suppose now that the $D\neq 0$ result is valid for $G_{n-1}$, the network 
of the $(n-1)$-site phosphorylation system. In order to visualize the calculations, 
we will reorder the rows and columns of $\Sys''_n$, placing $C(1), \, C(n+2)$, and 
$C(2n+3)$ as the leftmost columns, and $R(2)+R(2n+3), \, R(n+2)$, and 
$R(3n+3)$ as the uppermost rows. We notice that this ordering does 
not alter the sign of the determinants, hence we can write
\begin{equation}\label{eq:Dnonzero}
 D~=~\det\left(\left[\begin{array}{ccc|c}
                     0 & k_{\rm{cat}_0} & -l_{\rm{on}_0} & \cdots\\
                     k_{\rm{on}_0} & -k_{\rm{off}_0}-k_{\rm{cat}_0} &  0 & {\bf 0}\\
                     0 & 0 & l_{\rm{on}_0} & {\bf 0} \\
 \hline
                     {\bf 0} & {\bf 0} & {\bf 0} & B\\
                    \end{array}\right]\right)
~=~-k_{\rm{on}_0}k_{\rm{cat}_0}l_{\rm{on}_0}\det(B)~,
\end{equation}
\noindent
where $B$ is a $(3n-3)\times (3n-3)$-submatrix of $\Sys''_n$.
This matrix $B$ does not include either $C(1), C(n+2), 
C(2n+3)$, nor the first $(n+1)$-st or $(2n+1)$-st 
rows of $\Sys''_n$. We next will see how the matrix $B$ can be interpreted 
as the $3(n-1) \times 3(n-1)$-matrix $\Sys''_{n-1}$, the 
corresponding matrix of the smaller network $G_{n-1}$. This interpretation
will allow us to conclude 
by the inductive hypothesis that $D\neq 0$ in 
the positive orthant.

For the purpose of interpreting this submatrix of $\Sys''_n$ as the 
matrix of $G_{n-1}$, it is important to note that 
the deletion of $C(1), C(n+2)$, and $C(2n+3)$ from $\Sys''_n$ is 
equivalent to calculating $\Sys''_n$ after having deleted these columns 
from $A_{\kappa}^t$ before calculating $\Sys_n$. In turn, it is also 
equivalent to having deleted all the reactions that begin at the first, 
$(n+2)$-nd and $(2n+3)$-rd complexes of the network. Once we have 
additionally  
deleted the first, $(n+1)$-st, and $(2n+1)$-st  rows (i.e. $R(2)+R(2n+3)$, 
$R(n+2)$, and $R(3n+3)$), we obtain a new submatrix of $\Sys_n$ whose entries 
we can rename as follows: 
\begin{equation*}
k_{\rm{on}_j}=:k_{\rm{on}_{j-1}}',~ \, k_{\rm{off}_j}=:k_{\rm{off}_{j-1}}',~ \, k_{\rm{cat}_j}=:k_{\rm{cat}_{j-1}}',~
l_{\rm{on}_j}=:l_{\rm{on}_{j-1}}',~ \, l_{\rm{off}_j}=:l_{\rm{off}_{j-1}}',~ \, l_{\rm{cat}_j}=:l_{\rm{cat}_{j-1}}'~.
\end{equation*}
In fact, this new matrix is the corresponding \CSMatrix~$\Sys'_{n-1}$ for the network $G_{n-1}$, 
with corresponding rate constants indicated by primes. 
We can also establish a correspondence between the nodes of the two networks: letting $j'$ denote
the $j$-th node of $G_{n-1}$, then $j'$ corresponds to the following node of $G_n$:
\begin{equation*}
j' ~{\rm corresponds~to}\quad  \left\lbrace\begin{array}{ll}
  j+1 & \quad  {\rm if}~ 1 \leq j' \leq n \quad {\rm(complexes}~S_0+E, \dots, S_{n-1}+E~{\rm in}~G_{n-1}) \\
  j+2 &\quad  {\rm if}~ n+1 \leq j' \leq 2n \quad {\rm(complexes}~ES_0, \dots, ES_{n-2}~{\rm in}~G_{n-1}) \\
  j+3 &\quad  {\rm if}~ 2n+1 \leq j' \leq 3n-1 \quad {\rm(complexes}~S_0+F, \dots, S_{n-1}+F~{\rm in}~G_{n-1}) \\
  j+4 &\quad  {\rm if}~ 3n \leq j' \leq 4n-2 \quad {\rm(complexes}~FS_0, \dots, FS_{n-1}~{\rm in}~G_{n-1})~. \\
 \end{array}\right.
\end{equation*}
From this correspondence, it follows that $\det(B)$ equals $\det(\Sys''_{n-1})$, which is 
nonzero by inductive 
hypothesis, and therefore $D \neq 0$, which we wanted to prove.

We now complete the proof by verifying the following claim: 
the polynomial functions $D_{j },$ $D_{n+j+1 },$ $D_{2n+j+2}$ never vanish, and they all have 
the same constant sign as that of $D$ on $\Rplus^{6n}$ (for $1\leq j \leq n$).

We first prove this claim for the case $j=1$.  We again reorder the entries of the matrices 
as described above, and as this ordering does not alter the sign of the determinants, we can write:
\begin{small}
\begin{align*} 
D_{1}	
	~&=~  -\det\left(\left[\begin{array}{ccc|c}
                     l_{\rm{off}_0} & k_{\rm{cat}_0} & -l_{\rm{on}_0} & \cdots \\
                     0 & -k_{\rm{off}_0}-k_{\rm{cat}_0} & 0 & {\bf 0} \\
                     -l_{\rm{cat}_0}-l_{\rm{off}_0} & 0 & l_{\rm{on}_0} & {\bf 0} \\
 \hline
                     {\bf 0} & {\bf 0}  & {\bf 0}  & B\\
                    \end{array}\right]\right)
                    ~=~-(k_{\rm{off}_0}+k_{\rm{cat}_0})l_{\rm{on}_0}l_{\rm{cat}_0}\det(B)~,\\
D_{n+2}
	~&=~-\det\left(\left[\begin{array}{ccc|c}
                     0 & l_{\rm{off}_0} & -l_{\rm{on}_0} & \cdots\\
                     k_{\rm{on}_0} & 0 & 0 &{\bf 0} \\
                     0 & -l_{\rm{cat}_0}-l_{\rm{off}_0} & l_{\rm{on}_0} & {\bf 0} \\
 \hline
                     {\bf 0}  & {\bf 0}  & {\bf 0}  & B\\
                    \end{array}\right]\right)
~=~ -k_{\rm{on}_0} l_{\rm{on}_0} l_{\rm{cat}_0} \det(B)~,\\
D_{2n+3}
	~&=~-\det\left(\left[\begin{array}{ccc|c}
                     0 & k_{\rm{cat}_0} &  l_{\rm{off}_0} & \cdots\\
                     k_{\rm{on}_0} & -k_{\rm{off}_0}-k_{\rm{cat}_0} & 0 & {\bf 0} \\
                     0 & 0 & -l_{\rm{cat}_0}-l_{\rm{off}_0} & {\bf 0} \\
 \hline
                     {\bf 0}  &{\bf 0}  & {\bf 0}  & B\\
                    \end{array}\right]\right)
 ~=~ -k_{\rm{on}_0} k_{\rm{cat}_0} (l_{\rm{cat}_0}+l_{\rm{off}_0}) \det(B)~,
\end{align*}
\end{small}
where $B$ is the same matrix we described in equation~\eqref{eq:Dnonzero}. That is, $B=\Sys''_{n-1}$. As 
we already know that $D \neq 0$, we deduce that $\det(B)\neq 0$. By examining 
equation~\eqref{eq:Dnonzero} and the display above, we conclude that the claim is true for $j=1$.

For the $j > 1$ case, we will prove our claim by induction on $n$.  The base case is $n=2$ (as $j>1$ is not possible when $n=1$).  
In this case, the functions of interest are the following positive functions on $\mathbb{R}^{12}_{>0}$: 
$D=k_{\rm{on}_0}k_{\rm{cat}_0}l_{\rm{on}_0}k_{\rm{on}_1}k_{\rm{cat}_1}l_{\rm{on}_1}$,
$D_{2}=k_{\rm{on}_0}k_{\rm{cat}_0}l_{\rm{on}_0}(k_{\rm{off}_1}+k_{\rm{cat}_1})l_{\rm{on}_1}l_{\rm{cat}_1}$,
$D_{5}=k_{\rm{on}_0}k_{\rm{cat}_0}l_{\rm{on}_0}k_{\rm{on}_1}l_{\rm{on}_1}l_{\rm{cat}_1}$, and
$D_{8}=$ $k_{\rm{on}_0}k_{\rm{cat}_0}l_{\rm{on}_0}k_{\rm{on}_1}k_{\rm{cat}_1}(l_{\rm{cat}_1}+l_{\rm{off}_1})$.  
Hence our claim holds for $n=2$.

We now assume that the claim is true for $G_{n-1}$. As we did above, we view $G_{n-1}$ as a subgraph 
of $G_{n}$, and if we call $D'_{\ell(j')}$ the corresponding 
determinant of the $(n-1)$-site system (for 
{\small $\ell(j')=j',~ (n-1)+j'+1,~ 2(n-1)+j'+2, ~$} for {\small $1 \leq j' \leq n-1$}), then we have: 
\begin{small}
\begin{equation} \label{eq:D-induction}
D_{\ell(j) } \quad = \quad
	 (-1)^{(n+1)+1}k_{\rm{on}_0}(-1)^{1+n}k_{\rm{cat}_0}(-1)^{(2n-1)+(2n-1)}l_{\rm{on}_0}D'_{\ell(j')}
	\quad = \quad
	-k_{\rm{on}_0}k_{\rm{cat}_0}l_{\rm{on}_0}D'_{\ell(j')}~,
\end{equation}
\end{small}
for {\small $\ell(j')=j',~ (n-1)+j'+1,~ 2(n-1)+j'+2,$}  where {\small $1 \leq j' \leq n-1$}.
By the inductive hypothesis, the claim holds for the $D'_{\ell(j')}$, so by equation~\eqref{eq:D-induction}, the claim holds for the $D_{\ell(j)}$ as well.  This completes the proof.
\end{proof}

We now take care of the zero entries of the vectors $b^j$ defined in~\eqref{eq:basisKer}. 
We start by defining $D_{u \leftrightarrow v}$ as minus the determinant of the matrix 
obtained by replacing column $C(u)$ by $C(v)$ in $\Sys''_n$, for $1 \leq u \leq 3n+2$ such that 
$u \neq n+1,$ $u\neq 2n+2$, and $3n+3 \leq v \leq 4n+2$:
\begin{small}
\begin{equation}\label{eq:Duv}
D_{u \leftrightarrow v} 
	~:=~ -\det\left(\left[C(1) | \dots | \overset{\overset{u}{\downarrow}}{C(v)} | \dots | C(3n+2) \right]\right)~.
\end{equation}
\end{small}
We will deduce from the following lemma that $D_{u \leftrightarrow v}$ is equal to 
zero unless $u=j,$ $n+j+1,$ or $2n+j+2$ and $v=3n+j+2$, for $1 \leq j \leq n$. 

\begin{lemma}\label{lemma:MinorsZero}
Fix $j \in \{1,2, \dots, n\}$ and call $\widehat{\Sys'_n}$, the submatrix 
of $\Sys'_n$ obtained by deleting any two columns indexed by two elements 
of $I_j$. It holds that any $3n \times 3n$-minor of $\widehat{\Sys'_n}$ is equal to zero.
\end{lemma}

\begin{proof}
We will keep the notation $R(i)$ from the proof of Proposition~\ref{prop:Sign+Dnonzero}. 
We now prove the lemma first for $j=1$, then $j=n$, and then finally for $1 <j < n$.

For the case $j=1$, we focus on the reactions $1 \rightleftarrows n+2\stackrel{}{\rightarrow} 2$, 
$2n+3 \rightleftarrows 3n+3 \stackrel{}{\rightarrow} 2n+2$, and
$3n+4 \stackrel{}{\rightarrow} 2n+3$. If we delete $C(1)$ and $C(n+2)$, 
or $C(2n+3)$ and $C(3n+3)$, then the rows of 
$\widehat{\Sys'_n}$ corresponding to $R(n+2)$ or $R(3n+3)$ will be 
equal to zero and the minor will be zero.

If we delete $C(1)$ and $C(2n+3)$ (or $C(3n+3)$), or we delete $C(n+2)$ and 
$C(2n+3)$ (or $C(3n+3)$), the rows corresponding to $R(n+2)$ and $R(3n+3)$
will have only one entry different from zero and the determinant will be 
obviously zero if the column corresponding to any of this entries is not 
considered, or it will be the product of two constants and a $(3n-2) \times
(3n-2)$-minor that does not include the columns $C(1), C(n+2), C(2n+3), C(3n+3)$ nor 
the rows $R(n+2), R(3n+3)$.

It is important to notice that the columns of $A_{\kappa}^t$ carry the 
information of the reactions whose source (educt) is the corresponding complex, 
therefore, $C(\ell)$ carries the information of the reaction whose 
source is the $\ell$-th complex. As the only complexes that generate reactions 
whose product is the $(n+2)$-nd or $(3n+3)$-rd complexes are the first and 
$(2n+2)$ complexes, respectively, it follows that the columns that are being considered in 
this new $(3n-2) \times (3n-2)$-minor carry the information of reactions 
that do not end in either the $(n+2)$-nd or the $(3n+3)$-rd complexes. 
Hence the sum of the rows in this new submatrix, and therefore the minor as well, 
is equal to zero.

For $j=n$, the analysis is similar.

For $1<j<n$ we focus on the reactions $j \rightleftarrows n+j+1
\stackrel{}{\rightarrow} j+1$ and
 $2n+j+2 \rightleftarrows 3n+j+2
\stackrel{}{\rightarrow} 2n+j+1$. If we delete $C(j)$ and $C(n+j+1)$, 
or $C(2n+j+2)$ and $C(3n+j+2)$, then the rows of $\widehat{\Sys'_n}$ 
corresponding to $R(n+j+1)$ or $R(3n+j+2)$ will be equal to zero 
and the minor will be zero.

If we delete $C(j)$ and $C(2n+j+2)$ (or $C(3n+j+2)$), or we delete 
$C(n+j+1)$ and $C(2n+j+2)$ (or $C(3n+j+2)$), the rows corresponding 
to $R(n+j+1)$ and $R(3n+j+2)$ will have only one entry different from 
zero, and thus the determinant will be obviously zero if the column 
corresponding to any of these entries is not considered. Otherwise 
it will be the product of two nonzero rate constants and a 
$(3n-2)\times (3n-2)$-minor that does not include any of 
$C(j), C(n+j+1), C(2n+j+2), C(3n+j+2)$ nor any of $R(n+j+1), R(3n+j+2)$. 

But deleting these columns is equivalent to not considering the reactions 
whose sources (educts) are the complexes $j, n+j+1, 2n+j+2,$ or $3n+j+2$. This disconnects 
the graph into four linkage classes, so this new graph 
gives a Laplacian matrix formed by four blocks. The rows 
of $\Sys_n$ that we are considering in $\Sys'_n$ 
come from adding rows of the first and third blocks of $A_{\kappa}^t$, or the 
second and fourth ones; and the last rows of $\Sys_n$, which correspond to intermediary 
species, clearly belong to only one of the blocks. Then, this new submatrix 
of $\widehat{\Sys'_n}$ can be reordered into a two-block matrix, for which the 
sums of the rows in each block are zero. Hence, the matrix obtained from 
$\widehat{\Sys'_n}$ without these four columns and two rows has rank at most 
$3n-3$ and therefore any $(3n-2) \times (3n-2)$-minor will be zero.
\end{proof}

We are now ready to prove Theorem~\ref{thm:n-site}.

\begin{proof}[Proof of Theorem~\ref{thm:n-site}]

Due to Lemma~\ref{lemma:MinorsZero}, for a $3n \times 3n$-minor of
$\Sys'_n$ to be different from zero, we must obtain these $3n$ columns 
by choosing three from each group indexed by $I_j$, for $1 \leq j \leq n$.
In fact, any $3n \times 3n$-minor of $\Sys'_n$ that includes three columns 
from each group of four indexed by $I_j$, for $1 \leq j \leq n$, is always 
nonzero due to Proposition~\ref{prop:Sign+Dnonzero}.

We can now solve system~\eqref{eq:SysCramer} by applying Cramer's rule. 
Recall the notation from~\eqref{eq:Duv}:
\begin{equation*}
 \left[ \begin{array}{c}
  v_1\\
  \vdots\\
  v_{3n}\\
 \end{array}\right]= \frac{-1}{D} \left[\begin{array}{ccc}
    D_{1 \leftrightarrow 3n+3} & \dots & D_{1 \leftrightarrow 4n+2}\\
    \vdots &  & \vdots\\
    D_{3n+2 \leftrightarrow 3n+3} & \dots & D_{3n+2 \leftrightarrow 4n+2}\\
    \end{array}\right] \left[ \begin{array}{c}
				  v_{3n+1}\\
				  \vdots\\
				  v_{4n}\\
                                 \end{array}\right].
\end{equation*}
By Lemma~\ref{lemma:MinorsZero}, we already know that
in the $3n \times n$-matrix in the right-hand side above,
 the only 
nonzero entries are $D_{j},$ $D_{n+j+1}$, and $D_{2n+j+2}$. This gives us a 
description of $\ker(\Sys_n)$, which has a basis of the following form:
$$\{e_{n+1}, e_{2n+2}\} \cup \{b^1, b^2, \dots , b^n\}$$
for $b^j$ as in~\eqref{eq:basisKer}.

 This proves that the $n$-site phosphorylation system satisfies 
Condition~\ref{cond:1} for the partition $I_1, I_2, \dots, I_{n+2}$ and the 
basis of $\ker(\Sys_n)$, $\{b^1,b^2, \dots, b^n, e_{n+1}, e_{2n+2}\}$, 
described above.

We now prove that the $n$-site phosphorylation system additionally satisfies
 Conditions~\ref{cond:2} and~\ref{cond:3}. Condition~\ref{cond:2} is satisfied immediately by Proposition~\ref{prop:Sign+Dnonzero}.
With respect to Condition~\ref{cond:3}, we notice that the subspace spanned 
by the columns of the matrix $\Delta$ has the following basis:
\begin{equation}\label{eq:differences}
 \{ e_{2n+j+1}-e_j-e_{3n+2},~ e_{2n+j+1}-e_{n+j+1}, ~
 	e_{2n+j+1}-e_{j+1}-e_{3n+3}~|~1 \leq j \leq n\}.
\end{equation}
Therefore, the dimension of the image of $\Delta$ is $3n$, so $\ker(\Delta)~=~0$.
Hence, equation~\eqref{eq:kappa_cond} is trivially satisfied, as noted in Remark~\ref{rmk:condDelta}.

Then, by Theorem~\ref{th:toric}, it is immediate that the $n$-site phosphorylation 
system has toric steady states that are positive and real.  Finally, for a parametrization
of the steady state locus, 
let us consider the following matrix:
$$A~=~\left[\begin{array}{ccccc | cccc | cccc | cc}
             0 & 1 & 2 & \dots & n & 1 & 2 & \dots & n & 1 & 2 & \dots & n & 1 & 0 \\
             0 & 0 & 0 & \dots & 0 & 1 & 1 & \dots & 1 & 1 & 1 & \dots & 1 & 1 & 1 \\
             1 & 1 & 1 & \dots & 1 & 1 & 1 & \dots & 1 & 1 & 1 & \dots & 1 & 0 & 0 \\
            \end{array}\right] ~\in~ \R^{3 \times (3n+3)}.$$
It has maximal rank, and its kernel equals the span of
 all the differences $y_{j_2}-y_{j_1}$, for $j_1, j_2 \in I_j$, where 
$1 \leq j \leq n+2$, shown in~\eqref{eq:differences}.
After applying Theorem~\ref{thm:toric_param}, we are left to see that the point $\tilde{x}$ defined 
in the statement of the present theorem is a positive steady
state of the system.
But it is easy to check that $\tilde{x}$ is a positive steady state  by applying 
Theorem~\ref{th:binomial_ideal} to the following binomials:
$$D x_jx_{3n+2}-D_{j}x_{2n+j+1}, \, D x_{n+j+1}-D_{n+j+1} x_{2n+j+1}, 
\, D x_{j+1}x_{3n+3}-D_{2n+j+2} x_{2n+j+1}, \, \text{for}~ 1 \leq j \leq n.$$
This completes the proof.
\end{proof}


\section{Multistationarity for systems with toric steady states}
\label{sec:multistat}

In this section we focus on the capacity of a \CRS with toric steady states
to exhibit 
multiple steady states.  
Following prior work of Conradi~{\em et al.}
\cite{MAPK}
and Holstein \cite{Holstein}, 
we make use of an alternative notation for reaction 
systems to obtain a characterization of steady states 
(Proposition~\ref{prop:k_lambda}).  This result is used to prove a criterion for the
existence of multistationarity for systems with toric steady states that 
satisfy Conditions~\ref{cond:1}, \ref{cond:2}, and \ref{cond:3}
(Theorem~\ref{thm:multi_ss_toric}).  At the end of this section, we make the connection
to a related criterion of Feinberg.

Often a \CRS has a continuum of 
steady states, as long as one steady state exists.
However, as defined earlier (and as it is in Chemical Engineering),
multistationarity refers to the existence of multiple steady states
\emph{within one and the same} stoichiometric compatibility class. In
general one is interested in situations where the steady state locus
intersects a stoichiometric compatibility class in a finite number of
points \cite{fe95}. In Computational Biology one is sometimes
interested in situations where the steady state locus intersects an
affine subspace distinct from translates of  the stoichiometric subspace
$\St$ \cite{cc-flo-003}. Here we define multistationarity with respect
to a linear subspace in the following way. Consider a matrix
$Z\in\R^{s\times q}$, where $q$ is a positive integer. We say that the
\CRS $\dot x = \Sys \cdot \Psi(x)$ 
\emph{exhibits multistationarity with respect to the linear subspace
  $\ker(Z^t)$} if and only if there exist 
at least two distinct positive steady state vectors $x^{1}$, $x^{2}
\in \Rplus^s$ such that their difference lies in $\ker(Z^t)$; in other
words  the following equations must hold:
\begin{subequations}
  \begin{align}
    \label{eq:multistat_ode_x1}
    \Sys \cdot \Psi(x^1) ~&=~ 0 \\
    \label{eq:multistat_ode_x2}
    \Sys \cdot \Psi(x^2) ~&=~ 0 \\
    \label{eq:multistat_con_rel_x0_x1}
    Z^t\, x^1 ~&=~ Z^t\, x^2\ .
  \end{align}
\end{subequations}
Note that if the columns of $Z$ form a basis for $\St^\perp$, one
recovers the usual definition of multistationarity given in 
Section~\ref{sec:steadyStes}.
In this case, Equation~\eqref{eq:multistat_con_rel_x0_x1} 
states that the steady states $x^1$ and
$x^2$ belong to the same stoichiometric compatibility
class, and we simply speak of multistationarity, omitting the
linear subspace we are referring to.

\subsection{Second representation of a \CRS}
We now introduce a second representation of the differential equations
that govern a \CRS~(\ref{CRN}); this will prove useful for the characterization of steady
states (Proposition~\ref{prop:k_lambda}) and for establishing the
capacity of a \CRN for multistationarity.  
Letting $r$ denote the number
of reactions of a chemical reaction network $G$, we fix an ordering 
of these $r$ reactions and define the {\em
  incidence matrix} $\mathcal{I}\in\{-1,0,1\}^{m\times r}$ of the network 
to be the matrix whose $i$-th column has a $1$ in the row
corresponding to the product complex of the $i$-th reaction and a $-1$
for the educt (reactant) complex. Then the $(s \times r)$-matrix
product
\begin{equation}
  \label{eq:def_stoi_mat}
  N ~:=~ Y^t\, \mathcal{I}
\end{equation}
is known as the \emph{stoichiometric matrix}.
Thus, the $i$-th column of $N$ is the reaction vector corresponding
to reaction $i$.  
Next we define the
\emph{educt-complex matrix}
\begin{equation}
  \label{eq:def_Yl}
  \mathcal{Y} ~:=~ \left[ \tilde y_1,\, \tilde y_2, \, \ldots,\, \tilde y_r \right] ~,
\end{equation}
where the column $\tilde y_i$ of $\mathcal{Y}$ is defined
as the vector of 
the educt complex of the $i$-th reaction. Now we can define the vector
of educt complex monomials
\begin{equation}
  \label{eq:def_phi}
  \phi(x) ~:=~ \left(x^{\tilde y_1}, ~x^{\tilde y_2},~ \ldots, ~x^{\tilde y_r}\right)^t~.
\end{equation}
We also define $k \in \mathbb{R}^r_{>0}$ to be the vector of
reaction rate constants: $k_i$ is the rate constant of the $i$-th
reaction (that is, $k_i~=~\kappa_{i'j'}$ where the $i$-th reaction 
is from the complex $x^{y_{i'}}$ to $x^{y_{j'}}$).
We now give a second formulation for a \CRS~(\ref{CRN}) (cf.\
\cite{alg-002}):
\begin{equation}
  \label{eq:sigma_psi=N_k_phi}
  \dot x ~=~ N\, \diag(k)\, \phi(x)\ .
\end{equation}
  Both formulations of a \CRS given in equations~\eqref{CRN}
  and~\eqref{eq:sigma_psi=N_k_phi} 
  lead to the same system of ODEs and hence are equivalent.
  This can be made explicit by way 
  of the {\em doubling matrix} $D$ of dimension $m\times r$ which
  relates $\mathcal{Y}$ and $Y$ via  $\mathcal{Y} ~=~ Y^t\, D.$
Here the $i$-th column vector of $D$ is defined as the unit vector
$e_j$ of $\mathbb{R}^m$ such that $y_j$ is the educt (reactant)
complex vector of the $i$-th reaction.  
  From 
  \begin{displaymath}
    \dot x ~=~ N\, \diag(k)\, \phi(x) 
    ~=~ Y^t\,
    \mathcal{I}\, \diag(k)\, D^t\, \Psi(x)
    ~=~ \Sys\Psi(x) ~,
  \end{displaymath}
it follows that
  $\phi(x)~=~D^t\, \Psi(x)$ and $A_{\kappa}^{t}~=~\mathcal{I}\,
  \diag(k)\, D^t$.

\begin{example}\label{ex:N_k_phi}
  For the 1-site phosphorylation network~(\ref{eq:net_phospho_1_complete}),
  one obtains the matrices
\begin{small}
$$ 
\begin{array}{cc}
    \mathcal{I} \, = \,
    \left[
      \begin{array}{rrrrrr}
        -1 &  1 &  0 &  0 &  0 &  0 \\
        1 & -1 & -1 &  0 &  0 &  0 \\
        0 &  0 &  1 &  0 &  0 &  0 \\
        0 &  0  & 0 & -1 &  1 &  0 \\
        0 &  0  & 0 &  1 & -1 & -1 \\
        0 &  0  & 0 &  0 &  0 &  1
      \end{array}
    \right],
&
    D = \left[
      \begin{array}{rrrrrr}
        1 & 0 & 0 & 0 & 0 & 0 \\
        0 & 0 & 0 & 0 & 0 & 0 \\
        0 & 1 & 1 & 0 & 0 & 0 \\
        0 & 0 & 0 & 0 & 0 & 0 \\
        0 & 0 & 0 & 1 & 0 & 0 \\
        0 & 0 & 0 & 0 & 1 & 1 \\
      \end{array}
    \right], 
\end{array} 
$$
$$
    \mathcal{Y} = \left[
      y_1^t,\, y_3^t,\, y_3^t,\, y_5^t,\, y_6^t,\, y_6^t
    \right] \,= \, \left[
      \begin{array}{llllll}
        1 & 0 & 0 & 0 & 0 & 0 \\  
        0 & 0 & 0 & 1 & 0 & 0 \\  
        0 & 1 & 1 & 0 & 0 & 0 \\  
        0 & 0 & 0 & 0 & 1 & 1 \\  
        1 & 0 & 0 & 0 & 0 & 0 \\  
        0 & 0 & 0 & 1 & 0 & 0     
      \end{array}
    \right],
$$
\end{small}
  and the monomial vector $    \phi(x)~=~(x_1\, x_5,\, x_3,\, x_3,\, x_2\, x_6,\, x_4,\, x_4)^t$.
\end{example}

It follows from the differential equations~\eqref{eq:sigma_psi=N_k_phi}
that a positive concentration vector $x\in \mathbb{R}^s_{>0}$ is a
steady state for the \CRS defined by the positive reaction rate
constant vector $k$ if and only if
\begin{displaymath}
  \diag(k)\, \phi(x) ~ \in ~ \ker(N)\cap\Rplus^r ~ .
\end{displaymath}
We now recognize that the set $\ker(N)\cap\Rplus^r$, if nonempty,
 is the
relative interior of the pointed polyhedral cone $\ker(N)\cap\Rnn^r$.
To utilize this cone, we collect a finite set of {\em generators}
(also called ``extreme rays'') of the cone $\ker(N)\cap\Rnn^r$ as
columns of a non-negative matrix $\Egen$. Up to scalar multiplication,
generators of a cone are unique and form a finite set; as the cone of
interest arises as the intersection of an orthant with a linear
subspace, the generators are the vectors of the cone with minimal
support with respect to inclusion. (Background on polyhedral cones can
be found in the textbook of Rockafellar \cite{Rock}.) Letting $p$
denote the number of generators of the cone, we can use $\Egen$ to
express the condition for a positive vector $x\in \Rplus^s$ to be a
steady state of the \CRS in the following way:
\begin{equation}
  \label{eq:ss_condi}
  \diag(k)\, \phi(x) ~=~ \Egen\, \lambda~, \text{ for some }
  \lambda\in\Rnn^p\; \text{with}\; \Egen\, \lambda  \in \Rplus^r ~ .
\end{equation}
Note that this proves the following result which appears in \cite{MAPK}:
\begin{proposition}[Characterization of steady states of \CRSsNoSpace]
  \label{prop:k_lambda}    
  For a \CRN $G$, let $\Egen$ denote a corresponding generator matrix
  as defined above. Then a positive vector $x\in \Rplus^s$ is a steady
  state for the \CRS defined by reaction rate vector $k\in \Rplus^r$,
  if and only if there exists a vector $\lambda\in\Rnn^p$ such that
  \begin{equation}
    \label{eq:def_k_lam}
    k ~=~ \diag\left((\phi(x)\right)^{-1}\, \Egen\, \lambda~ \text{ and } \; 
    \Egen\, \lambda \in \Rplus^r  \ .
  \end{equation}
\end{proposition}

We now note that outside of a degenerate case, any positive 
concentration vector can be a steady state for 
appropriately chosen rate constants $k$.

\begin{remark} \label{rmk:degen_case}
 We now comment on the degenerate case of a network for which
 the set $\ker(N)\cap\Rplus^r$ is empty.  First, this case is
 equivalent to either of the following three conditions:
 (i) there is no positive dependence among the reaction vectors $(y_j-y_i)$, 
 (ii) the cone $\ker(N)\cap\Rnn^r$ is contained in a coordinate hyperplane, and
 (iii) the generator matrix $\Egen$ has at least one zero row.
  Now, in this degenerate case, it is clear that for any choice of reaction rate
 constants, the \CRS has no positive steady states.  This is because if 
 $x^* \in \mathbb{R}^s_{>0}$ is a steady state for the system with 
 reaction rate constants $\kappa_{ij}$, then the numbers 
 $\alpha_{ij}:= \kappa_{ij} \cdot (x^*)^{y_i}$ witness to the positive 
 dependence among the reaction vectors ($y_j-y_i)$'s. 
Outside of this degenerate case, it follows from
Proposition~\ref{prop:k_lambda} that there
  exists a vector of reaction rate constants $k$ for which the resulting
  \CRS has a positive
  steady state.
  Moreover, in this case any
  positive vector $x$ can be a steady state, by choosing $k$
  as in equation (\ref{eq:def_k_lam}) for some
  valid choice of $\lambda\in\Rnn^p$.
\end{remark}

Using our new notation, we return to the question of existence of steady states.
\begin{remark} \label{rmk:which_k}
  Recall the content of
    Corollary~\ref{cor:toric_network_same_partition}: for a \CRN for
    which a single partition works to satisfy Condition~\ref{cond:1}
    for all choices of positive rate constants, the set of rate
    constant vectors $k$ that yield systems with positive steady
    states is the semialgebraic set of $\Rplus^r$ defined by
    Conditions~\ref{cond:2} and~\ref{cond:3}.  We now note that
    Proposition~\ref{prop:k_lambda} implies that this set of rate
    constant vectors is the image of the following polynomial map:
    \begin{align*}
      \beta: \quad  \R_{>0}^s \times \Gamma ~& \to ~ \R_{>0}^r  \\
      (x, \lambda) ~& \mapsto ~  \diag(\phi(x))^{-1}\, \Egen\, \lambda~, 
    \end{align*}
    where $\Gamma~:=~ \{ \lambda \in \Rnn^p ~|~ \Egen \lambda \in
    \Rplus^r \}$.  
      In case that Condition~\ref{cond:1} holds
      and Condition~\ref{cond:3} is \textbf{trivially}
      satisfied (i.e.\ $\Delta$ has full row rank), the image of
      $\beta$ is cut out by the inequalities defined by Condition~\ref{cond:2}.

\end{remark}

\subsection{Main result on multistationarity}
We now make use of 
Proposition~\ref{prop:k_lambda} to examine which 
\CRSs with toric steady states exhibit multistationarity.
We first note that in the setting of Section~\ref{sec:toric}, 
the set of differences $\ln x^1 - \ln x^2$, where $x^1$ and $x^2$
are positive steady states for the same system, form a linear subspace.
As before, the notation ``$\ln x$'' for a vector $x \in \Rplus^s$ denotes
the vector $(\ln x_1, \ln x_2, \dots, \ln x_s)\in\mathbb{R}^s$; similarly we 
will make use of the notation ``$e^x$'' to denote
component-wise exponentiation.

  Our next theorem, the main result of this section, 
  is a consequence of \cite[Lemma~1]{MAPK}.  
  It states that a network that satisfies Condition~\ref{cond:1}
  has the capacity for
  multistationarity if and only if two subspaces, namely $\im(A^t)$
  and $\St$, both intersect non-trivially some (possibly lower-dimensional)
  orthant $\{ x \in \R^s ~|~ \sign(x) = \omega \}$ defined by a sign
  vector $\omega \in \{-,0,+ \}^s$.  We remark that this is a
  matroidal condition.  Related ideas appear in work of Feinberg~\cite{Fein95DefOne}, 
  and details on the connection between our work and Feinberg's appears at the end
  of this section.

\begin{theorem}[Multistationarity for networks with toric steady
  states]
  \label{thm:multi_ss_toric}
  
    Fix a \CRN $G$ with $s$ species and $m$ complexes,
    and let $Z\in\mathbb{Z}^{s\times q}$ be an integer matrix, for some positive
    integer $q$.
    Assume that the cone 
      $\ker(N)\cap\Rnn^r$ is not contained in any coordinate
      hyperplane.
    Assume moreover that 
    there exists a partition $I_1, I_2, \dots, I_d$ of the $m$ complexes of $G$ 
    such that Condition~\ref{cond:1} is satisfied for all rate constants.
  
  Recall the
  matrix $A$ for this partition from the proof of Theorem~\ref{thm:toric_param}. 
  Then there exists a reaction rate constant vector such that the resulting \CRS
    exhibits
    multistationarity with respect to the linear subspace
    $\ker(Z^t)$ if and only if there exists an orthant of $\mathbb{R}^s$ that
    both subspaces $\im(A^t)$
    and $\ker\left(Z^t\right)$ intersect nontrivially.
  More precisely, given nonzero vectors $\alpha\in\im(A^t)$
    and $\sigma\in\ker\left(Z^t\right)$ with
    \begin{equation}
      \label{eq:condi}
      \sign(\alpha)~=~\sign(\sigma) ~, 
    \end{equation}
  then two steady states $x^1$ and $x^2$ and a reaction rate constant vector $k$ 
    that witness multistationarity (that is, that satisfy 
    equations~\eqref{eq:multistat_ode_x1},~ 
    \eqref{eq:multistat_ode_x2}, and ~
    \eqref{eq:multistat_con_rel_x0_x1}) 
    arise in the following way:
    \begin{align}
      \label{eq:def_x1}
      \left(x^1_i\right)_{i=1,\, \ldots,\, s} ~&=~ 
      \begin{cases}
        \frac{\sigma_{i}}{e^{\alpha_{i}}-1}\text{, if $\alpha_{i} \neq
          0$} \\ 
        \bar{x}_i>0\text{, if $\alpha_{i} = 0$~,}
      \end{cases}
      \intertext{where $\bar{x}_i$ denotes an arbitrary positive number, 
        and}
      \label{eq:def_x2}
      x^2 ~&=~ \diag(e^\alpha)\, x^1 \\
      \label{eq:def_k}
      k ~&=~ \diag(\phi(x^1))^{-1}\, \Egen\, \lambda~ ,
    \end{align}
    for any non-negative vector $\lambda\in\Rnn^p$ for which 
    $ \Egen\, \lambda \in \Rplus^r$.  Conversely, any witness to multistationarity
    with respect to $\ker\left(Z^t\right)$ (given by some  
    $x^1$, $x^2\in\Rplus^s$, and $k\in\Rplus^r$)  
    arises from equations~ \eqref{eq:condi}, 
    \eqref{eq:def_x1},
    \eqref{eq:def_x2}, and
    \eqref{eq:def_k} for some 
    vectors $\alpha\in\im(A^t)$
    and $\sigma\in\ker\left(Z^t\right)$ that have the same sign.
  
\end{theorem}
\begin{proof}
  Assume that there exist nonzero vectors $\alpha\in\im(A^t)$
  and $\sigma\in\ker\left(Z^t\right)$ having the same sign.
  First note that the vectors $x^1$, $x^2$, and $k$ defined 
  by~\eqref{eq:def_x1},
  \eqref{eq:def_x2}, and
  \eqref{eq:def_k}, respectively, are positive because 
  $\alpha$ and $\sigma$ have the same sign and because 
  the cone $\ker(N)\cap\Rnn^r$ is not contained in a coordinate
  hyperplane. 
  By Proposition~\ref{prop:k_lambda}, equation~\eqref{eq:def_k} 
  implies that $x^1$ is a steady state of the system defined by 
  $k$.  We now claim that $x^2$ too is a steady state of the same
  system.  This follows from Theorem~\ref{thm:toric_param} because the 
  difference between $\ln x^1$ and $\ln x^2$ is in $\im(A^t)$:
  \begin{align*}
    \ln x^1 - \ln x^2 ~=~ -\alpha ~\in~ \im(A^t)~.
  \end{align*}
  Conversely, assume that vectors $x^1$, $x^2$, and $k$ 
  are a witness to multistationarity with respect to $\ker(Z^t)$.
  Let us now construct appropriate vectors $\alpha$ and 
  $\sigma$.  By Theorem~\ref{thm:toric_param}, the vector
  $\alpha ~:=~ \ln x^2 - \ln x^1$ is in $\im(A^t)$.  Next, we define
  $\sigma \in \R^s$ by $\sigma_i = (e^{\alpha_i}-1)x^1_i$
  if $\alpha_i \ne 0$ and $\sigma_i=0$ if $\alpha_i = 0$, so by 
  construction, $\alpha$ and $\sigma$ have the same sign.  In 
  addition, equations~\eqref{eq:def_x1} and ~\eqref{eq:def_x2}
  easily follow for these values of $\alpha$ and $\sigma$.  We also see
  that 
  \begin{align*}
    -\sigma~=~x^1-x^2 ~\in~ \ker(Z^t)~,
  \end{align*}
  so $\sigma \in \ker(Z^t)$.
  Finally, Proposition~\ref{prop:k_lambda} implies that there exists
  a valid $\lambda \in \Rnn^p$ that satisfies~\eqref{eq:def_k}.
\end{proof}

\begin{remark}

  If a \CRS defined by reaction rate constants $k^*$ and a partition
  of its complexes satisfy Conditions~\ref{cond:1}, \ref{cond:2}, and
  \ref{cond:3} (but not necessarily for other choices of rate
  constants), then the equations~\eqref{eq:condi}, \eqref{eq:def_x1}, 
  \eqref{eq:def_x2}, and
  \eqref{eq:def_k}
  in Theorem~\ref{thm:multi_ss_toric} still characterize  
  multistationarity.  In other words, $x^1$ and $x^2$ are two steady states 
  that demonstrate that the system defined by $k^*$ has the capacity for 
  multistationarity with respect to $\ker(Z^t)$ 
  if and only if there exist $\alpha \in \im(A^t)$,
  $\sigma \in \ker(Z^t)$, and $\lambda \in \Rnn^p$ such that those
  four equations hold.

\end{remark}

\begin{example}[Triangle network, continued] 

We return to the Triangle network analyzed in Examples~\ref{ex:toric_depends_on_rates} 
and~\ref{ex:return_triangle}.  The stoichiometric subspace is 
\[ 
\ker(\Sys) ~=~ \St ~=~ {\rm span}\{(1,-1)\}~.
\]
In the toric setting (recall that this is when $\kappa_{31}=\kappa_{32}$),
the partition for which the system satisfies Condition~\ref{cond:1} is $\{1,2\},\{3\}$, so a matrix $A$
for which
\[ 
\ker(A) ~=~ {\rm span}\{ y_2 - y_1 \} ~=~ {\rm span}\{(2,-2)\}
\]
is $A=[1~ 1]$.  We can see that the subspaces $\ker(Z^t)$
and $\im(A^t)={\rm span}\{ (1,1) \}$ do not both intersect
some orthant nontrivially.  So Theorem~\ref{thm:multi_ss_toric}
allows us to conclude that no system (for which 
$\kappa_{31}=\kappa_{32}$)
arising from the Triangle network exhibits multistationarity.
\end{example}

Although the capacity of the Triangle network to exhibit 
multistationarity is easily determined directly, without the need to apply
Theorem~\ref{thm:multi_ss_toric}, it is more difficult in the case of 
the multisite phosphorylation system.  
Recall that we proved in Theorem~\ref{thm:n-site}
that any $n$-site phosphorylation system satisfies Condition~\ref{cond:1}
with the same partition (for fixed $n$).  
Hence, Theorem~\ref{thm:multi_ss_toric}
 can be used to compute the semialgebraic set of reaction rate constants $k$
 that give rise to multistationarity for the phosphorylation networks.
 This was performed by  
Conradi~{\em et al.}\ (for the $2$-site network)~\cite{MAPK} and Holstein 
(for the general  $n$-site network)~\cite{Holstein}; multistationarity is possible
only for $n\geq 2$.  
Results on the number
of steady states of phosphorylation systems appeared in work of Wang and Sontag 
\cite{WangSontag} and is the focus of a forthcoming work of the authors 
\cite{article2}.

\subsection{Connection to related results on multistationarity}
We now make the connection between our results on the capacity of a \CRN to exhibit multistationarity 
and related results of Feinberg \cite{Fein95DefOne}.  
A {\em regular} network is a network for which 
(i)~$\ker(N)\cap\Rplus^r\neq \emptyset$, (ii)~each linkage class contains a unique terminal strong linkage class, 
and (iii)~removing the reaction(s) between any two adjacent complexes in a terminal strong linkage class disconnects 
the corresponding linkage class.  Recall from Remark~\ref{rmk:degen_case} that condition~(i) in this definition is 
simply the requirement that the reaction vectors $y_j-y_i$ are positively dependent, and that this condition is necessary for the existence of positive steady states.
Recall that the deficiency of a \CRN was discussed in \S~\ref{subsec:deficiency}.

We now can explain the relationship between Feinberg's result and ours.  Feinberg examined regular deficiency-one networks, 
while we are concerned with networks for which there exists a partition that satisfies Condition~\ref{cond:3} (for all rate constants).  
In these respective settings, both Theorem~4.1 and Corollary~4.1 of \cite{Fein95DefOne} and Theorem~\ref{thm:multi_ss_toric} 
in this article state that {\em a certain subset of $\mathbb{R}^s$ and the stoichiometric subspace both intersect the same orthant non-trivially 
if and only if the network has the capacity for multistationarity.}  In the result of Feinberg, this set is a union of certain 
polyhedral cones, while in our case, this set is the image of $A^t$.  In both cases, this set consists of all vectors
$\ln(c^*/c^{**})$, where $c^*$ and $c^{**}$ are steady states arising from the same rate constants.  
As an illustration, see Example~\ref{ex:n=1_again} below.

Let us now explain how the two results are complementary.  First, there are some networks for which only Feinberg's results apply.  
For example, consider any network for which the union of polyhedral cones obtained from Feinberg's results is not a linear space.  
Additionally, for some networks, only our results apply.  As an example, the $n>1$ multisite networks have deficiency greater than one.  
Finally, for some networks, both our results and Feinberg's apply, such as in the following example.

\begin{example} \label{ex:n=1_again}
The $1$-site phosphorylation network of Example~\ref{ex:n=1} is regular and has deficiency one.  In this case, both the image of $A^t$ 
and Feinberg's union of cones are the subspace of $\mathbb{R}^6$ spanned by the three vectors $(e_1+e_2+e_3+e_4)$, $(e_2+e_3+e_4+e_5)$, 
and $(e_3+e_4+e_5+e_6)$.  So in this instance, our Theorem~\ref{thm:multi_ss_toric} and Feinberg's Corollary~4.1 of \cite{Fein95DefOne} coincide.
\end{example}

Finally we note that the proofs of both results make use of special
structure of $\ker(\Sigma)$.  In our case, we assume the existence of a basis with
disjoint support.  For Feinberg's results, there is a non-negative
basis 
where the supports of the first $L$ basis vectors correspond
exactly to the $L$ terminal strong linkage classes, and the last basis
vector is the all-ones vector (here $L$ denotes the number of terminal
strong linkage classes). 


\medskip
\medskip

\noindent \textbf{{\large Acknowledgments:}}
We thank the Statistical and Applied Mathematical Sciences Institute (SAMSI),
USA, where this work was started. We are grateful to the organizers of the 2008-09 Program on Algebraic Methods in Systems 
Biology and Statistics at SAMSI, for generating the space for our interactions.  
We acknowledge two conscientious referees whose comments improved this article.



\end{document}